\documentclass{birkjour}
\def\st{{\mathfrak t}}
\def\dom{{\rm dom\,}}

\def\sgn{{\rm sgn\,}}
 \newtheorem{thm}{Theorem}[section]
 \newtheorem{cor}[thm]{Corollary}
 \newtheorem{lem}[thm]{Lemma}
 \newtheorem{prop}[thm]{Proposition}
 \theoremstyle{definition}
 
 \theoremstyle{remark}
 \newtheorem{rem}[thm]{Remark}
 \newtheorem{ex}{Example}
 \numberwithin{equation}{section}

\begin{document}

%
%
%
%
%
%
%
%
%

\title[Various form closures]{Various form closures
associated with a fixed non-semibounded self-adjoint operator}

\author[A. Fleige]{Andreas Fleige}

\address{%
Baroper Schulstrasse 27 a\\
44225 Dortmund\\
Germany}

\email{fleige.andreas@gmx.de}



\date{December 21, 2024}
\dedicatory{To the memory of Heinz Langer}

\begin{abstract}
If $T$ is a semibounded self-adjoint operator in a Hilbert space $(H, \, (\cdot , \cdot))$
then the closure of the sesquilinear form $(T \cdot , \cdot)$ is a unique Hilbert space completion.
In the non-semibounded case a closure is a Kre\u{\i}n space completion and generally, it is not unique.
Here, all such closures
are studied.
A one-to-one correspondence between all closed symmetric forms (with ``gap point'' $0$) and all
J-non-negative, J-self-adjoint and boundedly invertible Kre\u{\i}n space operators is observed.
Their eigenspectral functions are 
investigated, in particular near the critical point infinity.
An example for infinitely many closures 
of a fixed form $(T \cdot , \cdot)$ 
is discussed in detail
using a non-semibounded self-adjoint multiplication operator $T$ in a model Hilbert space.
These observations indicate that closed symmetric forms may carry more information than self-adjoint Hilbert space operators.
\end{abstract}

\maketitle

\section{Introduction}\label{sec_introduction}

It is a classical result that semibounded self-adjoint operators in Hilbert spaces
and closed semibounded symmetric sesquilinear forms
are in one-to-one correspondence and hence, include similar information.
In particular, if $T$ is a positive and boundedly invertible operator in the Hilbert space $(H, \, (\cdot , \cdot))$
then the associated closed form $\st [\cdot,\cdot]$ is obtained as the inner product of the Hilbert space completion
(or {\it form closure})
of the positive definite inner product $(T \cdot , \cdot)$
defined on the domain $\dom T$.
Note that this Hilbert space is continuously embedded in $(H, \, (\cdot , \cdot))$.
Conversely, $T$ is recovered from $\st [\cdot,\cdot]$ by Kato's First representation Theorem (see e.g. \cite[VI-§ 2.1]{Kato}).

A first approach to non-semibounded sesqulilinear forms
was initiated by McIntosh in \cite{Mc1, Mc2} (see also e.g. \cite{GKMV, CT}).
However, here we follow a different approach via Kre\u{\i}n space methods outlined e.g. in \cite{F7, FHS, FHSW2, FHSW3}. 
To this end consider a generally non-semibounded self-adjoint operator $T$ in $(H, \, (\cdot , \cdot))$
which is boundedly invertible.
Then, the inner product $(T \cdot , \cdot)$ on $\dom T$
allows a {\it Kre\u{\i}n space} completion, 
i.e. there is a Kre\u{\i}n space $(\dom \st, \, \st [\cdot,\cdot])$ 
continuously embedded  in $(H, \, (\cdot , \cdot))$ and
including $\dom T$ as a dense subspace such that $\st [\cdot,\cdot]$ coincides with $(T \cdot , \cdot)$ on $\dom T$
(cf. \cite[Section 5]{FHS}).
In the terminology of \cite{FHS} this means that $\st [\cdot,\cdot]$ is a {\it closure}
of $(T \cdot , \cdot)$ in $(H, \, (\cdot , \cdot))$
and obviously, this is a generalization of the positive setting from above.
Similarly, the terminology of a {\it closed} symmetric form is generalized to the non-semibounded situation.
Again, $T$ can be recovered from $\st [\cdot,\cdot]$ by a generalization of
Kato's First Representation Theorem (cf. \cite[Theorem 1]{F7} or \cite[Theorem 3.3]{FHS}).

However, a Kre\u{\i}n space completion is generally not unique.
In \cite{CL2} \'Curgus and Langer studied Kre\u{\i}n space completions in detail for a quite general setting.
In the present paper we investigate such completions from the point of view 
of non-semibounded form closures. 
In particular, we study all form closures of $(T \cdot , \cdot)$ for a fixed self-adjoint operator $T$ given as above.

So far, in the previous papers like \cite{F7, FHS, FHSW2, FHSW3}
the main focus was on so-called {\it regular} closed forms $\st [\cdot,\cdot]$
which are defined on $\dom |T|^{\frac{1}{2}}$
and hence, allow a generalization of the Second Representation Theorem
(cf. \cite[Theorem 3]{F7} or \cite[Theorem 4.2]{FHS}).
In particular, in \cite[Theorem 5.2]{FHS} the ``semibounded result'' from the beginning was generalized: 
There is a one-to-one correspondence
between all regular closed forms and all self-adjoint operators with a gap in the real spectrum.
In \cite{FHS} the above restriction to $0 \in \rho(T)$ is shifted to this gap and this is reflected by the introduction
of so-called ``gap points'' of a closed form.

In some sense
this ``one-to-one result'' is not complete since we know 
that there are also non-regular closed forms (see e.g. \cite[Theorem 6.9]{FHSW2}).
In the present paper our main interest is in non-regularity
where 
the form $\st [\cdot,\cdot]$ has no representation by means of $|T|^{\frac{1}{2}}$.
A different one-to-one correspondence is presented allowing all (not only regular) closed forms:

To this end we start with an arbitrary J-non-negative, J-self-adjoint and boundedly inverible operator $A$
in a {\it Kre\u{\i}n space} $(K, \, [\cdot,\cdot])$ and construct a {\it Hilbert space}
$(K_-, \, \{\cdot , \cdot\}_-)$
such that $K \subset K_-$ and the inclusion is dense and continuous.
Then, $\st [\cdot,\cdot] := [\cdot,\cdot]$ is a closed symmetric form in $(K_-, \, \{\cdot , \cdot\}_-)$ 
(with gap point $0$)
and $A$ appears as the range restriction of the representing self-adjoint operator 
in $(K_-, \, \{\cdot , \cdot\}_-)$
according to the generalized First Representation Theorem.
Now, the mapping $A \longrightarrow \st [\cdot,\cdot]$
defines a one-to-one correspondence between 
all J-non-negative, J-self-adjoint and boundedly inverible Kre\u{\i}n space operators
and all closed forms with a gap point $0$ (Theorem \ref{thm_1-1_J-non-neg}).

In the next step we consider the restriction of the inverse of this mapping to the form closures
of $(T \cdot , \cdot)$ with a fixed self-adjoint and boundedly invertible operator $T$ in
a Hilbert space $(H, \, (\cdot , \cdot))$ (Theorem \ref{thm_1-1_only_T}).
For the associated J-non-negative, J-self-adjoint and boundedly inverible operators infinity
is the only possible critical point
in the sense of Langer's theory of definitizable operators from \cite{L}.
There is precisely one {\it regular} closure $\st_0 [\cdot,\cdot]$
and this is characterized by the regularity of the critical point infinity of the associated Kre\u{\i}n space operator $A_0$ (Theorem \ref{thm_exceptual}).
This means that the eigenspectral function of $A_0$ according to \cite[Theorem II.3.1]{L}
is bounded near infinity.
Since for all other closures and operators we have non-regularity,
$\st_0 [\cdot,\cdot]$ and $A_0$ will be called {\it regularizations}
of the other closures and operators, respectively.
From the Kre\u{\i}n space point of view this regularization is a ``small'' modification of the operator and of the Kre\u{\i}n space as well.
Although this modification makes a big difference for the norm,
the eigenspectral functions themselves do not differ essentially.
More precisely, each eigenspectral function appears as the restriction of the spectral measure of $T$ to the form domain
(Theorem \ref{thm_spectralFunction}).
Furthermore, a consequence of \cite[Theorem 2.7]{CL2} is mentioned:
The form closure of $(T \cdot , \cdot)$ is unique if and only if $T$ is semibounded (Theorem \ref{thm_unique}).

Some of the above results do only make certain ideas more precise which had already been in the background of previous papers
(see again \cite{F7, FHS, FHSW2, FHSW3}).
However, when we leave this abstract level and turn to simple example operators
the difficulties increase in describing different form closures explicitly.
In a different framework such kind of problems is studied in \cite[Example 2.11, Proposition 4.2]{GKMV}.
Also from the Kre\u{\i}n space point of view so far there are not many examples of different completions known
(see e.g. \cite{H, CL2}).
Here, we first recall an example from \cite{F8}
where a non-regular closed form associated with an indefinite Sturm-Liouville operator is constructed.
Unfortunately, we cannot describe the regularization explicitly in this case.

In \cite[Theorem 5.2, proof part II]{CL2} infinitely many Kre\u{\i}n space completions are constructed on a certain abstract level.
Here, using some of these ideas 
we explicitly construct a family of infinitely many form closures or Kre\u{\i}n space completions
associated with the self-adjoint multiplication operator $(Tf)(x) := x f(x)$
in a model Hilbert space $L^2_{r_-}({\mathbb R})$
(Theorem \ref{thm_conclusions_KreinSpaceCompletion}, Corollary \ref{cor_ex_non-neg_closure}).
The inner product in this space is given by
\[
(f,g)_{r_-} := \int_{-\infty}^\infty f \overline{g} \, r_- \, dx, \quad (f, g \in L^2_{r_-}({\mathbb R}))
\]
where the weight function $r_- \in L_{loc}^1({\mathbb R})$ is non-negative, vanishes around $0$ and satisfies some additional conditions.
With the 
function $r(x) := x r_-(x)$ the regular closure of $(T \cdot , \cdot)_{r_-}$ is given by the Kre\u{\i}n space inner product
\[
\st_0 [f,g] := [f,g]_r := \int_{-\infty}^\infty f \overline{g} \, r \, dx, \quad (f, g \in \dom \st_0 := L^2_{r}({\mathbb R})).
\]
(Note that $r$ changes its sign.) Now, let $\alpha \in (0,2]$ and consider the functions
\[
\eta_{\alpha}(x) := (\sqrt{|x|^{\alpha} + 1} - \sqrt{|x|^{\alpha}})|r(x)|, \quad
\omega_{\alpha}(x) := \sqrt{|x|^{\alpha}} \, |r(x)| \quad (x \in {\mathbb R}).
\]
If we denote the even part of a function $f$ by $f_e$ and the odd part by $f_o$ then, for each $\alpha \in (0,2]$ an
additional non-regular closure can be defined on
\[
\dom \st_{\alpha} := \{f \in  L^2_{r_-}({\mathbb R}) \, | \, f_e \in  L^2_{\eta_{\alpha}}({\mathbb R}), \, f_o \in  L^2_{\omega_{\alpha}}({\mathbb R}) \} 
\]
by means of the Kre\u{\i}n space inner product
\[
\st_{\alpha}[f,g] := \lim_{k \rightarrow \infty} \int_{-k}^{k} f \overline{g} \, r  \, dx \quad (f, g \in \dom \st_{\alpha})
\]
(Proposition \ref{prop_t_alpha_KreinSpace}). 
It is observed that
this limit always exists.
Furthermore, for all $\alpha \in [0,2]$
the associated eigenspectral functions
are studied near infinity.

Finally, note that here we have an explicit example showing that
each of the infinitely many form closures carries more information ($r$ and $\alpha$)
than the original self-adjoint operator (only $r$).
This seems to imply that in the non-semibounded case 
the concept of closed symmetric sesquilinear forms is more suitable 
than the concept of self-adjoint Hilbert space operators. 
An interpretation in physics would be interesting.

\section{Preliminaries}\label{preliminaries}

For further use we recall some basic definitions and facts from known theories.

\subsection{J-non-negative and boundedly invertible operators in Kre\u{\i}n spaces}\label{subsec_preliminaries_definitizable}

An indefinite inner product space $(K, \, [\cdot , \cdot])$ is called a {\it Kre\u{\i}n space}
if it allows a so-called {\it fundamental decomposition} $K = K^+ \oplus K^-$
with a direct and orthogonal (with respect to $[\cdot , \cdot]$) sum of
a Hilbert space $(K^+, \, [\cdot , \cdot])$ and an anti Hilbert space $(K^-, \, [\cdot , \cdot])$.
If $P^{\pm}$ denotes the orthogonal projection on $K^{\pm}$ then $J := P^+ - P^-$
is called a {\it fundamental symmetry}
and $\{\cdot , \cdot\} := [J\cdot , \cdot]$ defines a positive definite inner product
such that $(K, \, \{\cdot , \cdot\})$ is a Hilbert space.
The induced topology also serves as the topology of the Kre\u{\i}n space.
A densely defined operator $A$ in the Kre\u{\i}n space $(K, \, [\cdot , \cdot])$ is called
{\it J-self-adjoint} or {\it J-non-negative} if $JA$ has the corresponding property in the Hilbert space $(K, \, \{\cdot , \cdot\})$.
Note that these definitions do not depend on the choice of $J$.
Furthermore, a {\it Kre\u{\i}n space} $(K, \, [\cdot , \cdot])$ is called a
{\it Kre\u{\i}n space completion} of an inner product space $(\widetilde{K}, \, [\cdot , \cdot\tilde{]})$
if $\widetilde{K}$ is a dense subspace of $K$ and $[\cdot , \cdot]$ and $[\cdot , \cdot\tilde{]}$
coincide on $\widetilde{K}$.
For more details on Kre\u{\i}n spaces we refer to \cite{AI}.

According to \cite{L} a J-self-adjoint operator $A$ is called {\it definitizable} if 
the resolvent set $\rho (A)$ is not empty
and there is a real so-called {\it definitizing polynomial} $p$ such that $p(A)$ is J-non-negative.
The zeroes of $p$ which also belong to the real part of the spectrum $\sigma (A) \cap {\mathbb R}$ 
are called {\it critical points of} $A$.
Additionally, $\infty$ is called a {\it critical point of} $A$
if $p$ has an odd degree (i.e. $p$ has a sign change at $\infty$) and $\sigma (A) \cap {\mathbb R}$ 
is neither bounded from above nor from below.

In the following we restrict ourselves to the case of a J-self-adjoint operator $A$
such that $A - \lambda$ is J-non-negative and boundedly invertible for some $\lambda \in {\mathbb R}$.
Then clearly, $A$ is definitizable with definitizing polynomial $p(t) = t - \lambda$.
Furthermore, $\infty$ is the only possible critical point of $A$.

Now, we recall Langer's result from \cite[Theorem II.3.1]{L} on eigenspectral functions
for the case of a J-self-adjoint, J-non-negative and boundedly invertible operator $A$ in a Kre\u{\i}n space $(K, \, [\cdot , \cdot])$
(i.e. $\lambda = 0$).
To this end let $\Sigma$ denote the semiring of all bounded intervalls and its complements in ${\mathbb R}$.
A mapping $E$ from $\Sigma$ to $L(K)$ (the space of all bounded operators in $(K, \, [\cdot , \cdot])$) is constructed
which is called the {\it eigenspectral function} of $A$.

For a bounded interval $\emptyset \neq \Delta \in \Sigma$ and $f \in K$ the limit
\begin{equation}\label{EigenspectralFunction}
E(\Delta)f := \lim_{\epsilon \searrow 0} \; \lim_{\delta \searrow 0} -\frac{1}{2 \pi i} \int_{C^{\delta}_{\Delta_{\epsilon}}} (A - \lambda)^{-1}f \; d \lambda
\end{equation}
exists  in $(K, \, [\cdot , \cdot])$ where $C^{\delta}_{\Delta_{\epsilon}}$ is defined for $\delta, \epsilon \in (0, 1)$ in the following way:
For $\Delta = [\alpha, \beta]$ with $-\infty < \alpha < \beta < \infty$ consider the positive oriented curve $C_{\Delta}$ in ${\mathbb C}$
consisting of the line segments which connect the points
\[
\beta + i, \quad \alpha + i, \quad \alpha - i, \quad \beta -  i, \quad \beta + i.
\]
In order to define $C^{\delta}_{\Delta}$ we take off from $C_{\Delta}$ the segments between
\[
\alpha + i \delta, \quad \alpha - i \delta, \quad \mbox{and} \quad \beta -  i \delta, \quad \beta + i \delta.
\]
For an arbitrary bounded interval $\emptyset \neq \Delta \in \Sigma$ with endpoints $\alpha \leq \beta$ put
\[
\Delta_{\epsilon} := [\alpha \mp \epsilon, \beta \pm \epsilon] \quad ( := \emptyset \mbox{ if } \alpha \mp \epsilon > \beta \pm \epsilon)
\]
where the upper (lower) sign must be taken whenever the corresponding endpoint is (not, respectively) a part of $\Delta$.
Furthermore, put $E(\emptyset) := 0$ and $E(\Delta) := I_K - E({\mathbb R} \setminus \Delta)$
for an unbounded $\Delta \in \Sigma$ where $I_K$ denotes the identity on $K$.
Then, the following Theorem is obtained from \cite[Theorem II.3.1]{L}
(see also 
\cite[Appendix B]{F2}).
\begin{thm}[Langer]\label{thm_EigenspectralFunction}
Let $A$ be a J-self-adjoint, J-non-negative and boundedly invertible operator in the Kre\u{\i}n space $(K, \, [\cdot , \cdot])$
and let $E(\Delta)$ be given as constructed above for $\Delta \in \Sigma$.
Then, $E$ defines a mapping from $\Sigma$ to $L(K)$
with the following properties ($\Delta, \Delta' \in \Sigma$):
\begin{eqnarray}
& & E(\Delta) \mbox{ is J-self-adjoint},                       \nonumber      \\
& & E(\Delta)E(\Delta') = E(\Delta \cap \Delta'),           \nonumber   \\
& & E(\Delta \cup \Delta')= E(\Delta) + E(\Delta')    \mbox{ if } \Delta \cup \Delta' \in \Sigma, \; \Delta \cap \Delta' = \emptyset,     \nonumber   \\
& & E({\mathbb R}) = I_K, \; E(\emptyset) = 0                  \nonumber      \\
& & E(\Delta)  \mbox{ is J-non-negative, if } \overline{\Delta} \subset (0,\infty),           \nonumber          \\
& & E(\Delta)  \mbox{ is J-non-positive, if } \overline{\Delta} \subset (-\infty,0),             \nonumber        \\
& & \mbox{if  } \Delta \mbox{ is bounded then }               \nonumber      \\
& & \quad E(\Delta)K \subset \dom A, \quad A E(\Delta) f =  E(\Delta) A f \quad (f \in \dom A)    \nonumber  \\
& & \quad \mbox{and the restriction } A|_{E(\Delta)K} \mbox{ is bounded},         \nonumber    \\
& & \sigma ( A|_{E(\Delta)K} ) \subset  \overline{\Delta}                            \nonumber
\end{eqnarray}
where $\sigma$ denotes the spectrum. 
\end{thm}
According to \cite[Section II.5]{L} $\infty$ is called a {\it singular critical point} of $A$
if for some $\varepsilon > 0$ and some $f \in K$
one of the limits
\begin{equation}\label{def_singularCritPt}
\lim_{\lambda \rightarrow \infty} E[\varepsilon, \lambda]f \quad \mbox{and} \quad \lim_{\lambda \rightarrow \infty} E[-\lambda, -\varepsilon]f
\end{equation}
does not exist with convergence in $(K, \, [\cdot , \cdot])$.
By \cite[Proposition II.5.6]{L} this is equivalent to the property that the operator norms
$||E[\varepsilon, \lambda]||$ or $||E[-\lambda, -\varepsilon]||$ are unbounded for $\lambda \rightarrow \infty$.

\subsection{Closed symmetric sesquilinear forms}\label{subsec_preliminaries_forms}

Let $\st [\cdot,\cdot]$ be a densely defined symmetric sesquilinear form
(short only {\it form})
in a Hilbert space $(H, \, (\cdot , \cdot))$ with domain $\dom \st$.
Then, $\st [\cdot,\cdot]$ is semibounded from below
if for some $\lambda \in {\mathbb R}$ the form
\begin{equation}\label{form_lambda}
\st_{\lambda} [\cdot,\cdot] := \st [\cdot,\cdot] - \lambda (\cdot,\cdot)
\end{equation}
is non-negative, i.e. $\st_{\lambda} [f,f] \geq 0$ for all $f \in \dom \st$.
Furthermore, according to \cite[Theorem VI 1.11]{Kato} the form $\st [\cdot,\cdot]$ is closed if and only if
$(\dom \st, \, \st_{\lambda}[\cdot,\cdot])$ is a Hilbert space which is continuously embedded in $(H, \, (\cdot , \cdot))$
for some $\lambda \in {\mathbb R}$.
The setting is similar if $\st [\cdot,\cdot]$ is semibounded from above, 
i.e. $-\st [\cdot,\cdot]$ is semibounded from below.

However, in the following we do not assume that $\st [\cdot,\cdot]$ is semibounded (from below or above).
Using (\ref{form_lambda}),
according to \cite{FHS} the form $\st [\cdot,\cdot]$ is called {\it closed}
if $(\dom \st, \, \st_{\lambda}[\cdot,\cdot])$ is a {\it Kre\u{\i}n space} which is continuously embedded in $(H, \, (\cdot , \cdot))$
for some $\lambda \in {\mathbb R}$ (called {\it gap point} of $\st [\cdot,\cdot]$).
A closed extension $\st [\cdot,\cdot]$ of a form $\tilde{\st} [\cdot,\cdot]$ is called a {\it closure} of $\tilde{\st} [\cdot,\cdot]$
if $\dom \tilde{\st}$ is dense in the Kre\u{\i}n space $(\dom \st, \, \st_{\lambda}[\cdot,\cdot])$ for some gap point $\lambda$.
These definitions do not depend on the choice of the gap point (cf. \cite[Lemma 3.1]{FHS})
and are obviously generalizations from the semibounded case.
We recall the following generalizations of Kato's
Representation Theorems \cite[Theorem VI-2.1, Theorem VI-2.23]{Kato}
according to \cite[Theorem 3.3, Theorem 4.2]{FHS}.
\begin{thm}[First Representation Theorem]\label{first}
Let $\st [\cdot,\cdot]$ be a 
closed symmetric
sesquilinear form in the Hilbert space $(H, \,  (\cdot,\cdot))$.
Then there exists a unique self-adjoint operator $T_{\st}$ in $(H,
\, (\cdot,\cdot))$ such that $\dom T_{\st} \subset \dom \st$ and
\[
\st[u, v] = (T_{\st} u, v) ,   \quad u \in \dom T_{\st}, \quad v \in
\dom \st.
\]
All gap points $\lambda$ of $\st [\cdot,\cdot]$ belong to the resolvent set of $T_{\st}$
and $\dom T_{\st}$ is dense in the Kre\u{\i}n space $(\dom \st, \, \st_{\lambda}[\cdot,\cdot])$.
\end{thm}
Associated with a closed form $\st [\cdot,\cdot]$ is also the range restriction $A_\st$
\begin{equation}
\label{rangerestriction} 
\dom A_{\st} := T_{\st}^{-1}(\dom \st), \qquad A_\st f := T_{\st} f \quad (f \in \dom A_{\st})
\end{equation}
of $T_{\st}$ to $\dom \st$. 
If $\lambda \in {\mathbb R}$ is a gap point of $\st [\cdot,\cdot]$
then by \cite[Lemma 4.1]{FHS} $A_\st$ 
is $J$-self-adjoint and $A_{\st} - \lambda$ is J-non-negative and boundedly invertible in the Kre\u{\i}n
space $(\dom \st, \, \st_\lambda [\cdot,\cdot])$.
Consequently, $A_{\st}$ is definitizable and $\infty$ is the only possible critical point of $A_{\st}$.
\begin{thm}[Second Representation Theorem]\label{second}
Let $\st [\cdot,\cdot]$ be a closed symmetric
sesquilinear form in the Hilbert space $(H, \,  (\cdot,\cdot))$
and let $T_{\st}$ and $A_\st$ be the associated operators. Then
\begin{equation}\label{regular}
\dom \st = \dom |T_\st|^\frac{1}{2}
\end{equation}
if and only if $\infty$ is not a singular critical point of $A_\st$.
In this case the topology of the Kre\u{\i}n space 
$(\dom \st, \, \st_\lambda [\cdot,\cdot])$
is induced by the inner product
$(|T_\st - \lambda|^\frac{1}{2} \; \cdot \; ,|T_\st -
\lambda|^\frac{1}{2} \; \cdot \; )$  for each gap point $\lambda \in {\mathbb R}$
\end{thm}
According to \cite{FHS} a closed symmetric form $\st [\cdot,\cdot]$ is said to be
\textit{regular} if (\ref{regular}) is satisfied. 
In this case also a representation of $\st [\cdot,\cdot]$ by means of $|T_\st|^\frac{1}{2}$
in analogy to the classical Second Representation Theorem \cite[Theorem VI-2.23]{Kato}
can be found in \cite[Theorem 4.2]{FHS}.
The following result from \cite[Theorem 5.2]{FHS}
was already mentioned in the Introduction:
\begin{thm}\label{one-to-one_regular}
The mapping $\st[\cdot,\cdot] \to T_{\st}$ defines a one-to-one
correspondence between all regular 
closed symmetric forms in $(H,\,(\cdot,\cdot))$ and all self-adjoint operators in
$(H,\,(\cdot,\cdot))$ with spectrum different from the whole real
axis ${\mathbb R}$.
\end{thm}
By \cite[Proposition 2.5]{FHSW3} the definition of regularity can be weakened:
\begin{prop}\label{form_subset}
Let $\st [\cdot,\cdot]$ be a closed symmetric
sesquilinear form in the Hilbert space $(H, \,  (\cdot,\cdot))$
and let $T_{\st}$ be the associated operator. Then the following statements are equivatent:
\begin{enumerate}
\item[(i)] $\dom \st = \dom |T_\st|^\frac{1}{2}$,
\item[(ii)] $\dom \st \subset \dom |T_\st|^\frac{1}{2}$,
\item[(iii)] $\dom \st \supset \dom |T_\st|^\frac{1}{2}$.
\end{enumerate}
\end{prop}
This result can be improved by a slight extension of \cite[Lemma 2.4]{FHSW3}:
\begin{lem}\label{two_forms_coincide}
If we have two closed forms $\st_1 [\cdot,\cdot]$ and $\st_2 [\cdot,\cdot]$
associated with the same self-adjoint representing operator $T$
 in the Hilbert space $(H, \,  (\cdot,\cdot))$
then $\dom \st_1 \subset \dom \st_2$ implies 
$\st_1 [\cdot,\cdot] = \st_2 [\cdot,\cdot]$.
\end{lem}
\begin{proof}
From \cite[Lemma 2.4]{FHSW3} we can conclude that $\dom \st_1 = \dom \st_2$
and (looking in the proof of \cite[Lemma 2.4]{FHSW3})
that the embeddings of the associated Kre\u{\i}n spaces are continuous (using two gap points).
Now, let $f \in \dom \st_1 \, (= \dom \st_2), \, f_n \in \dom T$ such that $f_n \rightarrow f \, (n \rightarrow \infty)$
with respect to this Kre\u{\i}n space topology.
Then, according to Theorem \ref{first} we have
\[
\st_1 [f,f] = \lim_{n \rightarrow \infty} \st_1 [f_n,f_n] = \lim_{n \rightarrow \infty} (Tf_n,f_n) = \lim_{n \rightarrow \infty} \st_2 [f_n,f_n] = \st_2 [f,f].
\]
This implies $\st_1 [\cdot,\cdot] = \st_2 [\cdot,\cdot]$ by the polarization identity.
\end{proof}

\section{A space triplet associated with a J-non-negative, J-self-adjoint and boundedly invertible Kre\u{\i}n space operator}\label{sec_triplet}

Here a construction from \cite[Section 4.1]{F2} is recalled for 
the general case of a J-non-negative, J-self-adjoint and boundedly invertible operator $A$  in a Kre\u{\i}n space $(K, \, [\cdot , \cdot])$
(rather than for an indefinite  Kre\u{\i}n-Feller operator).

Let $J$ be a fundamental symmetry of $(K, \, [\cdot , \cdot])$ 
and denote the associated Hilbert space inner product by $\{\cdot , \cdot\} := [J \cdot , \cdot]$.
Furthermore, put
\[
K_+ := \dom (JA)^{\frac{1}{2}}, \qquad \{f,g\}_+ := \{(JA)^{\frac{1}{2}}f, (JA)^{\frac{1}{2}}g\} \quad (f,g \in K_+).
\]
Then $(K_+, \, \{\cdot , \cdot\}_+)$ is a Hilbert space which is dense and continuously embedded in $(K, \, [\cdot , \cdot])$.
Note that by \cite[Remark 1.5]{C1} the form $\{\cdot , \cdot\}_+$ on $K_+$ 
coincides with the closure of the positive definite sesquilinear form $[A \cdot , \cdot]$
defined on $\dom A$.
Let $K_-$ be the dual space of $(K_+, \, \{\cdot , \cdot\}_+)$, i.e. the space of all linear functionals 
$\varphi : K_+ \longrightarrow {\mathbb C}$ which are continuous with respect to $\{\cdot , \cdot\}_+$.
This is a complex Banach space with
\[
(\varphi + \psi)(f) := \varphi(f) + \psi(f), \quad (\alpha \cdot \varphi)(f) := \overline{\alpha} \cdot \varphi(f) \quad (f \in K_+)
\]
and with the norm
\[
\{\varphi\}_- := \sup_{f \in K_+, \{f,f\}_+ \leq 1} |\varphi(f)|.
\]
Each element $f \in K$ defines a unique continuous linear functional $[\cdot, f]$ on $K_+$.
In this sense we have the inclusion $K \subset K_-$ and the inner product $[\cdot , \cdot]$ on $K$
extends naturally to
\[
[\cdot , \cdot] : \, K_+ \times K_- \longrightarrow {\mathbb C}, \quad [f, \varphi] := \varphi (f) \quad (f \in K_+, \, \varphi \in K_-).
\]
Now, let $A_-$ denote the operator from the Riesz Representation Theorem, i.e.
\[
A_- g := \{\cdot, g\}_+, \quad (g \in \dom A_- := K_+).
\]
This is an isometric isomorphism between $(K_+, \, \{\cdot , \cdot\}_+)$
and the dual space $(K_-, \, \{\cdot\}_-)$. By definition we have for $g \in \dom A \, (\subset K_+)$
\[
[f,A_-g] = \{f, g\}_+ = \{(JA)^{\frac{1}{2}}f, (JA)^{\frac{1}{2}}g\} = [f, Ag] \quad (f \in K_+)
\]
and hence, $A_-g = Ag$. Therefore, $A$ is the restriction of $A_-$ to $\dom A$
and the inverse $A^{-1}$ is the restriction of $A_-^{-1}$ to $K$.
By the isometric property of $A_-$ we have for $\varphi \in K_-$
\[
[A_-^{-1}\varphi,\varphi] = [A_-^{-1}\varphi, A_- (A_-^{-1}\varphi)] = \{A_-^{-1}\varphi, A_-^{-1}\varphi\}_+ = \{\varphi\}_-^2.
\]
Therefore, $K_-$ turns into a Hilbert space with the inner product
\[
\{\varphi,\psi\}_- := [A_-^{-1}\varphi,\psi] \, (= \{A_-^{-1}\varphi, A_-^{-1}\psi\}_+) \quad (\varphi,\psi \in K_-).
\]
\begin{lem}\label{lem_A_-}
\begin{enumerate}
\item[(i)] $(K, \, [\cdot , \cdot])$ is dense and continuously embedded in the Hilbert space $(K_-, \, \{\cdot , \cdot\}_-)$.
\item[(ii)] Considered in the Hilbert space $(K_-, \, \{\cdot , \cdot\}_-)$, the operator $A_-$ is self-adjoint and boundedly invertible.
\item[(iii)] Considered in the Hilbert space $(K_+, \, \{\cdot , \cdot\}_+)$, the range restriction $A_+$ 
\[
\dom A_+ := A^{-1}(K_+), \quad A_+ f := Af \quad (f \in \dom A_+)
\]
is self-adjoint and boundedly invertible.
\item[(iv)] The operator $A_-$ allows the representation
\[
[f,g] = \{A_-f,g\}_-  \qquad (f \in \dom A_-, \, g \in K).
\]
\end{enumerate}
\end{lem}
\begin{proof}
(i) Let $\varphi \in K_-$ be orthogonal to $K$ with respect to $\{\cdot , \cdot\}_-$, i.e.
\[
0 = \{\varphi , f\}_- = [A_-^{-1}\varphi, f] \quad \mbox{for all} \quad f \in K.
\]
Then, $A_-^{-1}\varphi \in K_+$ is orthogonal to $K$ with respect to $[\cdot , \cdot]$.
Consequently, we have $A_-^{-1}\varphi = 0$ and hence, $\varphi = 0$.
This implies the density of $K$ in $(K_-, \, \{\cdot , \cdot\}_-)$.
Furthermore, for $f \in K$ we can estimate
\[
\{f , f\}_-^2 = [A^{-1}f, f]^2 \leq \{A^{-1}f, A^{-1}f\}\{f,f\} \leq ||A^{-1}|| \{f,f\}^2
\]
where $||A^{-1}||$ denotes the operator norm of $A^{-1}$
in $(K, \, \{\cdot , \cdot\})$.
Therefore, the embedding is continuous.

(ii) First, for $f, g \in \dom A_- = K_+$ we observe
\[
\{A_-f,g\}_- = [A_-^{-1}(A_-f),g] = [f,g] = \overline{\{A_-g,f\}_-} = \{f, A_-g\}_- .
\]
Therefore, $A_-$ is symmetric with respect to $\{\cdot , \cdot\}_-$ and hence, self-adjoint since the range of $A_-$ is the whole space $K_-$.

(iii) Similarly, for $f, g \in \dom A_+$ we observe
\[
\{A_+f,g\}_+ = \{Af,g\}_+ = [Af, A_-g]  = [Af, Ag] = \overline{\{A_+g,f\}_+} = \{f,A_+g\}_+.
\]
Therefore, $A_+$ is symmetric with respect to $\{\cdot , \cdot\}_+$ and hence, self-adjoint since the range of $A_+$ is the whole space $K_+$.

(iv) immediately follows from the definition of $\{\cdot , \cdot\}_-$.
\end{proof}
Now, we have constructed the ``space triplet''
\begin{equation}\label{triplet}
K_+ \subset K \subset K_-
\end{equation}
where each inclusion is dense and continuous with respect to the associated topologies.
Finally, as an example, we present ``model spaces'' which appear e.g. as the images
of Fourier transformations associated with indefinite Sturm-Liouville operators
or more generally, of indefinite Kre\u{\i}n-Feller operators (see e.g. \cite[Section 4.4, Section 4.5]{F2}).
\begin{ex}\label{ex_model_spaces}
Let the real function $r \in L_{loc}^1({\mathbb R})$ satisfy the sign conditions
\begin{equation}\label{function_r_sign}
r(x) = 0 \; \mbox{a.e. on} \; [-\varepsilon,\varepsilon],
\quad x r(x) > 0 \; \mbox{a.e. on} \; (-\infty,-\varepsilon) \cup (\varepsilon, \infty)
\end{equation}
with some $\varepsilon > 0$ and assume that $r$ is odd, i.e.
\begin{equation}\label{function_r_odd}
r(-x) = - r(x) \; \mbox{a.e. on} \; {\mathbb R}.
\end{equation}
Furthermore, consider the functions
\begin{equation}\label{function_r_pm}
r_+(x) := x r(x),
\quad r_-(x) := \frac{1}{x}  r(x) \qquad (x \in {\mathbb R} \setminus \{0\})
\end{equation}
which are well defined a.e. on ${\mathbb R}$ and non-negative.
On the weighted space $L^2_r({\mathbb R})$
we define the following inner products
\begin{equation}\label{inner_prod_r}
(f,g)_r := \int_{-\infty}^\infty f \overline{g} \, |r| \, dx, \quad 
[f,g]_r := \int_{-\infty}^\infty f \overline{g} \, r \, dx \quad (f, g \in L^2_r({\mathbb R}))
\end{equation}
and on $L^2_{r_\pm}({\mathbb R})$ 
\[
(f,g)_{r_\pm} := \int_{-\infty}^\infty f \overline{g} \, r_\pm \, dx \quad (f, g \in L^2_{r_\pm}({\mathbb R})).
\]
Then, $(L^2_{r_+}({\mathbb R}), \, (\cdot,\cdot)_{r_+})$ and $(L^2_{r_-}({\mathbb R}), \, (\cdot,\cdot)_{r_-})$
are Hilbert spaces.
Furthermore, $(L^2_r({\mathbb R}), \, [\cdot,\cdot]_r)$ is a Kre\u{\i}n space
with the fundamental symmetry 
\begin{equation}\label{fundSymm}
(J f)(x) := \sgn (x) f(x) \quad (x \in {\mathbb R})
\end{equation}
satisfying $[J\cdot,\cdot]_r = (\cdot,\cdot)_r$ and hence, $(L^2_{r}({\mathbb R}), \, (\cdot,\cdot)_{r})$ is the associated Hilbert space.
From (\ref{function_r_sign}) we can conclude the inclusions
\begin{equation}\label{triplet_model}
L^2_{r_+}({\mathbb R}) \subset L^2_{r}({\mathbb R}) \subset L^2_{r_-}({\mathbb R})
\end{equation}
and each inclusion is dense and continuous with respect to the associated topologies.
Moreover, $L^2_{r_-}({\mathbb R})$ can be regarded as the dual space 
of the Hilbert space $(L^2_{r_+}({\mathbb R}), \, (\cdot,\cdot)_{r_+})$
if we identify each element $g \in L^2_{r_-}({\mathbb R})$ with the functional $[\cdot,g]_r$ on $L^2_{r_+}({\mathbb R})$,
i.e. with
\[
\varphi_g (f) := [f,g]_r := \int_{-\infty}^\infty f \overline{g} \, r \, dx \quad (f \in L^2_{r_+}({\mathbb R})).
\]
(Here, we use an obvious extension of the inner product $[\cdot,\cdot]_{r}$ 
from (\ref{inner_prod_r}) to $L^2_{r_+}({\mathbb R}) \times L^2_{r_-}({\mathbb R})$.)
Note, that indeed this integral exists and the functional $\varphi_g$ is continuous with respect to $(\cdot,\cdot)_{r_+}$ since
\begin{eqnarray}
|\varphi_g (f)|^2 & \leq & \left( \int_{-\infty}^\infty \sqrt{|x|} |f(x)| |g(x)| \frac{1}{\sqrt{|x|}} \, |r(x)| \, dx \right)^2    \nonumber \\
& \leq & \left(\int_{-\infty}^\infty |x| |f(x)|^2  \, |r(x)| \, dx\right)\left(\int_{-\infty}^\infty |g(x)|^2 \frac{1}{|x|} \, |r(x)| \, dx\right)   
\nonumber  \\
& = &
(f,f)_{r_+}(g,g)_{r_-}.   \nonumber
\end{eqnarray}
Next, in the Kre\u{\i}n space $(L^2_r({\mathbb R}), \, [\cdot,\cdot]_r)$ we define the operator $A$ by
\begin{eqnarray}
\dom A & := & \{f \in L^2_r({\mathbb R}) \, | \, \int_{-\infty}^\infty |x|^2 |f(x)|^2  \, |r(x)| \, dx < \infty \},  
\nonumber \\
(Af)(x) & := & x f(x)  \qquad (f \in \dom A)   \nonumber
\end{eqnarray}
and in the Hilbert spaces $(L^2_{r_\pm}({\mathbb R}), \, (\cdot,\cdot)_{r_\pm})$ the operators $A_{\pm}$ by
\begin{eqnarray}
\dom A_{\pm} & := & \{f \in L^2_{r_\pm}({\mathbb R}) \, | \, \int_{-\infty}^\infty |x|^2 |f(x)|^2  \, r_{\pm}(x) \, dx < \infty \},  
\nonumber \\
(A_{\pm}f)(x) & := & x f(x)  \qquad (f \in \dom A_{\pm}).   \nonumber
\end{eqnarray}
Then, by (\ref{function_r_sign}) $A$ is J-non-negative, J-self-adjoint and boundedly invertible and $A_+$ and $A_-$ are self-adjoint and boundedly invertible
(all with respect to the corresponding spaces).
In particular, we have $\dom A_- =  L^2_{r_+}({\mathbb R})$ 
and also $\dom (JA)^{\frac{1}{2}} =  L^2_{r_+}({\mathbb R})$
and for $f,g \in  L^2_{r_+}({\mathbb R})$
\[
((JA)^{\frac{1}{2}}f, (JA)^{\frac{1}{2}}g)_r 
= \int_{-\infty}^\infty \sqrt{|x|} f(x) \sqrt{|x|} \; \overline{g(x)} \, |r(x)| \, dx
= (f,g)_{r_+}.
\]
Finally, we observe that with 
\begin{equation}\label{triplet_r}
K:= L^2_r({\mathbb R}), \; [\cdot,\cdot] := [\cdot,\cdot]_r \quad 
\mbox{ and } \quad  K_\pm:= L^2_{r_\pm}({\mathbb R}), \; \{\cdot,\cdot\}_{\pm} := (\cdot,\cdot)_{r_{\pm}}
\end{equation}
and with the associated operators $A$ and $A_\pm$
the general setting of Section \ref{sec_triplet} is given and (\ref{triplet_model}) coincides with the space triplet (\ref{triplet}) in this ``model situation''.
\end{ex}

\section{Relations between closed symmetric sesquilinear forms and J-self-adjoint, J-non-negative operators in Kre\u{\i}n spaces}\label{sec_forms_operators}

In Theorem \ref{one-to-one_regular} the one-to-one relation between all regular closed symmetric forms
and all self-adjoint Hilbert space operators with a gap in its real spectrum was recalled from \cite{FHS}.
Now, we get rid of the restriction to {\it regular} closed forms
and in return we consider J-non-negative Kre\u{\i}n space operators.
Using the construction from Section \ref{sec_triplet} we shall see that
all closed symmetric forms with a gap point $0$ are in one-to-one correspondance
with all J-non-negative, J-self-adjoint and boundedly invertible Kre\u{\i}n space operators.

This result is now presented in the following Theorem \ref{thm_1-1_J-non-neg}.
Here, we use certain identifications
without stating them explictly in the Theorem in order to avoid a too technical overhead.
Two Hilbert spaces with closed forms 
$(H_1, \, (\cdot,\cdot)_1, \, \st_1 [\cdot,\cdot])$ and $(H_2, \, (\cdot,\cdot)_2, \, \st_2 [\cdot,\cdot])$
(both with gap point $0$)
are identified if there is an isometric Hilbert space isomorphisms $\phi$ from $(H_1, \, (\cdot,\cdot)_1)$ to $(H_2, \, (\cdot,\cdot)_2)$
such that $\dom \st_2 = \phi(\dom \st_1)$ and $\st_1 [f,g] = \st_2 [\phi(f), \phi(g)]$ for all $f,g \in \dom \st_1$.
Similarly, two Kre\u{\i}n spaces with J-non-negative, J-self-adjoint and boundedly invertible operators
$(K_1, \, [\cdot,\cdot]_1, \, A_1)$ and $(K_2, \, [\cdot,\cdot]_2, \, A_2)$
are identified if there is an isometric Kre\u{\i}n space isomorphisms $\psi$ from $(K_1, \, [\cdot,\cdot]_1)$ to $(K_2, \, [\cdot,\cdot]_2)$
such that $\psi(\dom A_1) = \dom A_2$ and $A_1 = \psi^{-1} A_2 \psi$.
In this sense the tuples in the formulation of the following Theorem must be regarded as the corresponding equivalence classes.

\begin{thm}\label{thm_1-1_J-non-neg}
The mapping $\Phi$ from the set
\begin{eqnarray}
{\mathcal M} := \{(H, \, (\cdot,\cdot), \, \st [\cdot,\cdot]) & | & (H, \, (\cdot,\cdot)) \mbox{ Hilbert space, } \st [\cdot,\cdot] \mbox{ closed symmetric}    \nonumber \\
& & \mbox{form in this space with a gap point } 0  \}          \nonumber
\end{eqnarray}
to the set
\begin{eqnarray}
{\mathcal N} := \{(K, \, [\cdot,\cdot], \, A) & | & (K, \, [\cdot,\cdot]) \mbox{ Kre\u{\i}n space, } A \mbox{ J-non-negative, J-self-}                            \nonumber \\
& & \mbox{adjoint and boundedly invertible in this space}  \}          \nonumber
\end{eqnarray}
given by the rule
\begin{equation}\label{rule_Phi}
K := \dom \st, \quad [\cdot,\cdot] := \st [\cdot,\cdot], \quad A := A_{\st}
\end{equation}
(using the range restriction $A_{\st}$ from (\ref{rangerestriction})) is bijective. 
Its inverse $\Phi^{-1}$ is given by the rule
\begin{equation}\label{rule_Phi_inverse}
H := K_-, \quad  (\cdot,\cdot) :=  \{\cdot,\cdot\}_-, \quad \dom \st := K, \quad \st [\cdot,\cdot] := [\cdot,\cdot]
\end{equation}
(using the construction from Section \ref{sec_triplet}).
\end{thm}
\begin{proof}
Starting with $(H, \, (\cdot,\cdot), \, \st [\cdot,\cdot]) \in {\mathcal M}$
it is clear by Section \ref{subsec_preliminaries_forms} that $(K, \, [\cdot,\cdot], \, A)$ given by 
(\ref{rule_Phi}) belongs to ${\mathcal N}$.
Conversely, we show that the Hilbert space $(K_-, \, \{\cdot,\cdot\}_-)$ given by 
Section \ref{sec_triplet}
is isometrically isomorphic to the original space 
$(H, \, (\cdot,\cdot))$ such that also $[\cdot,\cdot]$ and $\st [\cdot,\cdot]$ are connected by this isomorpism.
To this end, first observe that the space $K_+$ according to Section \ref{sec_triplet} 
coincides with the space $\dom T$ where $T = T_{\st}$ is the associated representing operator.
Indeed, $\dom T$ is a Hilbert space with the inner product $(T \cdot,T \cdot)$ since $0 \in \rho(T_{\st})$
and on $\dom A$ this inner product coincides with $\st [A \cdot,\cdot]$.
Therefore, $(T \cdot,T \cdot)$ on $\dom T$ is the closure of $\st [A \cdot,\cdot]$ 
and hence, $(T \cdot,T \cdot)$ coincides with the inner product $\{\cdot,\cdot\}_+$ on $\dom T \, (= K_+)$
according to Section \ref{sec_triplet}.
Now, for $f \in H$ consider the linear functional 
$\phi_f := (T \cdot , f)$ on $\dom T$.
Obviously, this functional is continuous with respect to $\{\cdot,\cdot\}_+ \, (= (T \cdot,T \cdot))$
and for $f \in K \, (= \dom \st)$ we have $\phi_f = \st [\cdot, f]$.
Consequently, by $\phi : \, f \rightarrow \phi_f$ the function
$\phi$ maps $H$ into the dual space $K_-$ of $(K_+, \, \{\cdot,\cdot\}_+)$
and on $K$ it coincides with the embedding of $K$ into $K_-$ according to Section \ref{sec_triplet}.
Furthermore, for $f \in H$ we have
\[
\{\phi_f\}_- = \sup_{g \in K_+, \{g,g\}_+ \leq 1} |\phi_f(g)| = \sup_{g \in \dom T, \, (Tg,Tg) \leq 1} |(Tg, f)| = \sqrt{(f,f)}
\]
since $0 \in \rho(T)$.
Therefore, $\phi$ maps $H$ isometrically into a (closed) subspace of $K_-$
and since $K$ is dense in $K_-$ the range of $\phi$ is the whole space $K_-$.
Consequently, $\phi$ is the required isomorphism.

On the other hand, starting with $(K, \, [\cdot,\cdot], \, A) \in {\mathcal N}$
it is clear by Section \ref{sec_triplet} that $(H, \, (\cdot,\cdot), \, \st [\cdot,\cdot])$ given by (\ref{rule_Phi_inverse}) belongs to ${\mathcal M}$.
Conversely, with (\ref{rule_Phi}) we return to the original element $(K, \, [\cdot,\cdot], \, A) \in {\mathcal N}$ by definition
since by Lemma \ref{lem_A_-} $A_-$ is the representing operator for the form $\st [\cdot,\cdot] = [\cdot,\cdot]$
and $A$ is its range restriction to $K$.

It remains to show that $\Phi$ maps the equivalence classes
according to the indicated identifications into each other and hence, the mapping $\Phi$ is well defined.
To this end, first consider two Hilbert spaces with closed forms 
$(H_1, \, (\cdot,\cdot)_1, \, \st_1 [\cdot,\cdot])$ and $(H_2, \, (\cdot,\cdot)_2, \, \st_2 [\cdot,\cdot])$
(both with gap point $0$)
which are identified by means of an isomorphisms $\phi \, : \, H_1 \rightarrow H_2$. 
Then, by $\st_1 [f,g] = \st_2 [\phi(f), \phi(g)]$ it is clear that also the Kre\u{\i}n spaces
$(\dom \st_1, \, \st_1 [\cdot,\cdot] )$ and $(\dom \st_2, \, \st_2 [\cdot,\cdot] )$
are isometrically isomorph by means of $\phi$.
Furthermore, with the representing operators $T_1 = T_{\st_1}$ and $T_2 = T_{\st_2}$ we have
\[
\st_2 [\phi(f), \phi(g)] = \st_1 [f,g] = (T_1 f, g)_1 =  (\phi(T_1 f), \phi(g))_2
\]
for all $f \in \dom T_1, \, g \in \dom \st_1$
which implies $\phi(f) \in \dom T_2$ and $T_2 \phi(f) = \phi(T_1 f)$.
This gives $T_1 \subset \phi^{-1} T_2 \phi$ and hence, $T_1 = \phi^{-1} T_2 \phi$ by the self-adjointness.
Consequently also the range restrictions satisfy $A_{\st_1} = \phi^{-1} A_{\st_2} \phi$
and therefore, also the Kre\u{\i}n spaces and its operators must be identified.

On the other hand consider two Kre\u{\i}n spaces with J-non-negative, J-self-adjoint and boundedly invertible operators
$(K_1, \, [\cdot,\cdot]_1, \, A_1)$ and $(K_2, \, [\cdot,\cdot]_2, \, A_2)$
which are identified by means of an isomorphisms $\psi \, : \, K_1 \rightarrow K_2$
such that $A_1 = \psi^{-1} A_2 \psi$.
Then, also the positive sesquilinear forms $[A_1 \cdot,\cdot]_1$ and $[A_2 \cdot,\cdot]_2$
are connected by $\psi$, i.e. $[A_2 \psi \cdot,\psi \cdot]_2 = [A_1 \cdot,\cdot]_1$
and hence, for its closures 
we have ${K_2}_+ = \psi ({K_1}_+)$
and $\{\cdot,\cdot{\}_1}_+ = \{\psi \cdot,\psi \cdot{\}_2}_+$.
Therefore, $\varphi$ is a continuous liner functional on $({K_2}_+, \, \{\cdot,\cdot{\}_2}_+)$ (i.e. $\varphi \in {K_2}_-$)
if and only if $\varphi \circ \psi$ is a continuous liner functional on $({K_1}_+, \, \{\cdot,\cdot{\}_1}_+)$ (i.e. $\varphi \circ \psi \in {K_1}_-$)
and we have
\[
\{\varphi{\}_2}_- = \sup_{g \in {K_2}_+, \{g,g{\}_2}_+ \leq 1} |\varphi(g)| = \sup_{h \in {K_1}_+, \{h,h{\}_1}_+ \leq 1} |\varphi( \psi (h))| = \{\varphi \circ \psi{\}_1}_- .
\]
Consequently, $\varphi \rightarrow \varphi \circ \psi$ defines an isometric isomorhism $\psi_-$ 
from the Hilbert space $({K_2}_-, \, \{\cdot,\cdot{\}_2}_-)$ to $({K_1}_-, \, \{\cdot,\cdot{\}_1}_-)$
such that the corresponding forms satisfy
\[
[\psi_-(f),\psi_-(g)]_1 = [f \circ \psi,g \circ \psi]_1 = [f,g]_2
\]
for all $f,g \in K_2 \, (= \psi (K_1))$.
Then, also these Hilbert spaces and forms must be identified.
\end{proof}

With a shift we can get rid of the restriction to the gap point $0$.
\begin{cor}\label{cor_1-1_shifted}
Let $\lambda \in {\mathbb R}$.
Then, the mapping $\Phi_{\lambda}$ from the set
\begin{eqnarray}
\{(H, \, (\cdot,\cdot), \, \st [\cdot,\cdot]) & | & (H, \, (\cdot,\cdot)) \mbox{ Hilbert space, } \st [\cdot,\cdot] \mbox{ closed symmetric}   \nonumber \\
& & \mbox{ form in this space with gap point } \lambda  \}          \nonumber
\end{eqnarray}
to the set
\begin{eqnarray}
\{(K, \, [\cdot,\cdot], \, A) & | & (K, \, [\cdot,\cdot]) \mbox{ Kre\u{\i}n space, } A \mbox{ J-self-adjoint, }    A - \lambda                        \nonumber \\
& & \mbox{ J-non-negative and boundedly invertible in this space}  \}          \nonumber
\end{eqnarray}
given by the rule
\[
K := \dom \st, \quad [\cdot,\cdot] := \st [\cdot,\cdot] -  \lambda (\cdot,\cdot), \quad A := A_{\st}
\]
(using the range restriction $A_{\st}$ from (\ref{rangerestriction})) is bijective. 
Its inverse $\Phi^{-1}_{\lambda}$ is given by the rule
\[
H := K_-, \quad  (\cdot,\cdot) :=  \{\cdot,\cdot\}_-, \quad \dom \st := K, \quad \st [\cdot,\cdot] := [\cdot,\cdot] + \lambda \{\cdot,\cdot\}_-
\]
(using the construction from Section \ref{sec_triplet} for $A - \lambda$).
\end{cor}
\begin{proof}
A form $\st [\cdot,\cdot]$ is closed with gap point $\lambda$ if and only if $\st_{\lambda} [\cdot,\cdot]$ 
from (\ref{form_lambda}) is closed with gap point $0$.
In this case the representing operators $T_{\st}$ and $T_{\st_{\lambda}}$ are connected via $T_{\st_{\lambda}} = T_{\st} - \lambda$.
This implies $A_{\st_{\lambda}} = A_{\st} - \lambda$ for the corresponding range restrictions.
Therefore, the statement follows from Theorem \ref{thm_1-1_J-non-neg} applied to $\st_{\lambda} [\cdot,\cdot]$ and $A_{\st_{\lambda}}$.
\end{proof}
Note that this Section can be regarded as a more detailed discussion of a setting from \cite[Section 5]{FHS},
in particular from \cite[Proposition 5.3]{FHS}.

\section{Families of form closures and of Kre\u{\i}n space completions associated with a fixed form $(T\cdot,\cdot)$}\label{sec_closures}

First, we mention the following improvement of \cite[Proposition 2.5]{FHSW3} 
(where a similar statement was obtained only for {\it regular} closed forms).
\begin{prop}\label{prop_gap_point_resolvent}
The set of gap points of a closed form $\st [\cdot,\cdot]$ in a Hilbert space $(H, \, (\cdot,\cdot))$
coincides with the real part of the resolvent set of its representing operator, i.e. with ${\mathbb R} \cap \rho(T_{\st})$.
\end{prop}
\begin{proof}
By the First Representation Theorem \ref{first} each gap point belongs to $\rho(T_{\st})$.
Conversely, let $\lambda_0 \in {\mathbb R} \cap \rho(T_{\st})$
and let $\lambda_1 \in {\mathbb R}$ be a gap point of $\st [\cdot,\cdot]$.
Then, we show that $\st_{\lambda_0} [\cdot,\cdot]$ defines a Kre\u{\i}n space structure on $\dom \st$
with the Hilbert space topology induced by $\st_{\lambda_1} [\cdot,\cdot]$,
i.e. also $\lambda_0$ is a gap point of $\st [\cdot,\cdot]$.

We have $\lambda_0, \lambda_1 \in \rho(T_{\st})$ for the representing operator $T_{\st}$ 
and hence, also for its range restriction $A := A_{\st}$ the operators
$A - \lambda_0$ and $A - \lambda_1$ are boundedly invertible in the Kre\u{\i}n space $(\dom \st, \, \st_{\lambda_1}[\cdot,\cdot])$.
Let $J_{\lambda_1}$ denote a fundamental symmetry in this Kre\u{\i}n space 
and let $\st_{\lambda_1}(\cdot,\cdot) := \st_{\lambda_1}[J_{\lambda_1}\cdot,\cdot]$ be the associated Hilbert space inner product.
Then, for $f, g \in \dom \st$ we can calculate
\begin{eqnarray}
\st_{\lambda_0} [f,g]  & = &  \st_{\lambda_0} [(A - \lambda_1) (A - \lambda_1)^{-1}f,g]                       \nonumber \\
& = &  \st_{\lambda_0} [A (A - \lambda_1)^{-1}f,g] - \lambda_1 \st_{\lambda_0} [(A - \lambda_1)^{-1}f,g]            \nonumber  \\
& = &  \st [A (A - \lambda_1)^{-1}f,g] - \lambda_0 (A (A - \lambda_1)^{-1}f,g)            \nonumber  \\
& & \quad  
- \lambda_1 \st [(A - \lambda_1)^{-1}f,g] + \lambda_1 \lambda_0 ((A - \lambda_1)^{-1}f,g)          \nonumber  \\
& = &  \st [A (A - \lambda_1)^{-1}f,g] - \lambda_0 \st [(A - \lambda_1)^{-1}f,g]            \nonumber  \\
& & \quad  
- \lambda_1 (A(A - \lambda_1)^{-1}f,g) + \lambda_1 \lambda_0 ((A - \lambda_1)^{-1}f,g)          \nonumber  \\
& = &  \st [(A - \lambda_0) (A - \lambda_1)^{-1}f,g] - \lambda_1 ((A - \lambda_0)(A - \lambda_1)^{-1}f,g)         \nonumber  \\
& = &  \st_{\lambda_1} [(A - \lambda_0) (A - \lambda_1)^{-1}f,g]           \nonumber  \\
& = & \st_{\lambda_1} (J_{\lambda_1}(A - \lambda_0) (A - \lambda_1)^{-1}f,g) .          \nonumber
\end{eqnarray}
By the symmetry of $\st_{\lambda_0} [\cdot,\cdot]$ the operator $G := J_{\lambda_1}(A - \lambda_0) (A - \lambda_1)^{-1}$
is symmetric and everywhere defined in the Hilbert space $(\dom \st, \, \st_{\lambda_1}(\cdot,\cdot))$
and hence, $G$ is self-adjoint and bounded. 
Furthermore, $G$ is invertible and its inverse $G^{-1} = (A - \lambda_1) (A - \lambda_0)^{-1}J_{\lambda_1}$ is also bounded.
Therefore, by \cite[1.6.13]{AI} $\st_{\lambda_0}[\cdot,\cdot] = \st_{\lambda_1}(G \cdot,\cdot)$ 
defines a suitable 
Kre\u{\i}n space inner product
on $(\dom \st, \, \st_{\lambda_1}(\cdot,\cdot))$.
\end{proof}
Now, we consider a fixed Hilbert space $(H, \, (\cdot,\cdot))$
and a fixed self-adjoint and boundedly invertible operator $T$ in this space. 
Then, Proposition \ref{prop_gap_point_resolvent} allows the following characterization of
the restriction of the mapping $\Phi$ from Theorem \ref{thm_1-1_J-non-neg}
to the family of all closed forms $\st [\cdot,\cdot]$ in $(H, \, (\cdot,\cdot))$ which are represented by $T$, i.e. $T_{\st} = T$.
\begin{thm}\label{thm_1-1_only_T}
Let $T$ be a self-adjoint and boundedly invertible operator in the Hilbert space $(H, \, (\cdot,\cdot))$.
Then, the mapping $\Phi_T$ from the set
\begin{eqnarray}
{\mathcal M}_T = \{\st [\cdot,\cdot] & | & \st [\cdot,\cdot] \mbox{ is a closure of } (T\cdot,\cdot) \mbox{ in } (H, \, (\cdot,\cdot))  \}          \nonumber
\end{eqnarray}
to the set
\begin{eqnarray}
{\mathcal N}_T = \{(K, \, [\cdot,\cdot], \, A) & | & (K, \, [\cdot,\cdot]) \mbox{ is a Kre\u{\i}n space completion of }      \nonumber  \\
& &  (\dom T, \, (T\cdot,\cdot)) \mbox{  and continuously embedded in }   
\nonumber  \\
& & 
(H, \, (\cdot,\cdot)), \; A = T|_{\dom A} \mbox{ with } \dom A = T^{-1}(K) \}          \nonumber
\end{eqnarray}
given by the rule (\ref{rule_Phi}) is bijective.
We have 
\[
{\mathcal N}_T \subset {\mathcal N}, \qquad 
{\mathcal M}_T = \{\st [\cdot,\cdot] \; | \; (H, \, (\cdot,\cdot), \, \st [\cdot,\cdot]) \in {\mathcal M}, \; T_{\st} = T \}
\]
where the sets ${\mathcal N}$ and ${\mathcal M}$ are given by Theorem \ref{thm_1-1_J-non-neg}.
\end{thm}
\begin{proof}
The statements are clear by definition and by the First Representation Theorem \ref{first} when we observe that
$T$ is the representing operator of each
closure $\st [\cdot,\cdot]$ of the sesquiliear form $(T\cdot,\cdot)$
and hence,  by Proposition \ref{prop_gap_point_resolvent} $\st [\cdot,\cdot]$ has the gap point $0$.
\end{proof}
The following observation is a consequence of Theorem \ref{thm_1-1_J-non-neg}:
\begin{rem}\label{rem_1-1_only_T}
If we start with a closure
$\st [\cdot,\cdot] \in {\mathcal M}_T$ of $(T\cdot,\cdot)$ in $(H, \, (\cdot,\cdot))$
and consider the element $(K, \, [\cdot,\cdot], \, A) := \Phi_T (\st [\cdot,\cdot]) \in {\mathcal N}_T$
and finally construct the associated Hilbert space $(K_-, \, \{\cdot,\cdot\}_-)$ according to Section \ref{sec_triplet}
then $(H, \, (\cdot,\cdot))$ and $(K_-, \, \{\cdot,\cdot\}_-)$ are isometrically isomorphic.
Furthermore, from the proof of Theorem \ref{thm_1-1_J-non-neg} we can conclude
that also the associated operators $T$ in $(H, \, (\cdot,\cdot))$ and $A_-$ in $(K_-, \, \{\cdot,\cdot\}_-)$
are connected by this isomorpism, say $\phi : H \rightarrow K_-$: We have
$T = \phi^{-1}A_-\phi$.
\end{rem}
In the next step, 
the regularity of forms is reformulated in order to
identify precisely one exceptional element in the family ${\mathcal M}_T$ and one in ${\mathcal N}_T$.
\begin{thm}\label{thm_exceptual}
Let $T$ be a self-adjoint and boundedly invertible operator in the Hilbert space $(H, \, (\cdot,\cdot))$.
Then, the following statements hold true 
using the sets ${\mathcal M}_T$ and ${\mathcal N}_T$ and the mapping $\Phi_T$
from Theorem \ref{thm_1-1_only_T}:
\begin{enumerate}
\item[(i)] One closure $\st_0 [\cdot,\cdot] \in {\mathcal M}_T$ is regular. 
It is characterized by the property $\dom \st_0 = \dom |T|^{\frac{1}{2}}$. 
All other closures are not regular.
\item[(ii)] For one element $(K_0, \, [\cdot,\cdot]_0, \, A_0) \in {\mathcal N}_T$ infinity is not a singular critical point
of the J-self-adjoint, J-non-negative and boundedly invertible operator $A_0$.
For all others infinity is a singular critical point.
\item[(iii)] We have $(K_0, \, [\cdot,\cdot]_0, \, A_0) = \Phi_T (\st_0 [\cdot,\cdot])$, i.e.
the two exceptional elements in ${\mathcal M}_T$ and ${\mathcal N}_T$ are connected via (\ref{rule_Phi}).
\item[(iv)] For two closures $\st_1 [\cdot,\cdot], \st_2 [\cdot,\cdot] \in {\mathcal M}_T$
the relation $\dom \st_1 \subset \dom \st_2$ implies $\st_1 [\cdot,\cdot] = \st_2 [\cdot,\cdot]$.
\item[(v)] For two Kre\u{\i}n space completions $(K_1, \, [\cdot,\cdot]_1), (K_2, \, [\cdot,\cdot]_2)$ of $(\dom T, \, (T\cdot,\cdot))$
which are continuously embedded in $(H, \, (\cdot,\cdot))$
the relation $K_1 \subset K_2$ implies $K_1 = K_2$ and $[\cdot,\cdot]_1 = [\cdot,\cdot]_2$.
\end{enumerate}
\end{thm}
\begin{proof}
Again, in view of Proposition \ref{prop_gap_point_resolvent} the
statements are clear by definition and by Section \ref{subsec_preliminaries_forms},
in particular, by Lemma \ref{two_forms_coincide}.
\end{proof}
Theorem \ref{thm_exceptual}
justifies to call the regular closure $\st_0 [\cdot,\cdot] \in {\mathcal M}_T$ according to Theorem \ref{thm_exceptual} (i)
the {\it regularization} of each closure $\st [\cdot,\cdot] \in {\mathcal M}_T$.
Similarly, among all J-self-adjoint, J-non-negative and boundedly invertible operators 
which are connected with $T$ by ${\mathcal N}_T$ the exceptional operator will be called {\it regularization}.
More precisely, we call the exceptional element $(K_0, \, [\cdot,\cdot]_0, \, A_0) \in {\mathcal N}_T$ according to Theorem \ref{thm_exceptual} (ii)
the {\it regularization} of each 
element $(K, \, [\cdot,\cdot], \, A) \in {\mathcal N}_T$.

Now, we 
reformulate some of the above statements for a fixed J-self-adjoint, J-non-negative and boundedly invertible operator in a Kre\u{\i}n space.
\begin{cor}\label{cor_regularization}
Let $A$ be a J-self-adjoint, J-non-negative and boundedly invertible operator in a Kre\u{\i}n space $(K, \, [\cdot,\cdot])$.
Let the Hilbert space $(H, \, (\cdot,\cdot))$ and the form $\st [\cdot,\cdot]$ be given by (\ref{rule_Phi_inverse}).
Furthermore, let $T := T_{\st}$ be the associated self-adjoint operator in $(H, \, (\cdot,\cdot))$.
Then, for all elements $(\widetilde{K}, \, [\cdot,\cdot\tilde{]}, \, \widetilde{A}) \in {\mathcal N}_T$ 
the spaces $(\widetilde{K}, \, [\cdot,\cdot\tilde{]})$ have $\dom T^2$ as a common dense subspace 
and the operators $\widetilde{A}$ coincide on this subspace:
\begin{equation}\label{common_subset}
\dom T^2 \subset \dom \widetilde{A}, \qquad \widetilde{A} f = T f \quad (f \in \dom T^2).
\end{equation}
For all but one of the elements $(\widetilde{K}, \, [\cdot,\cdot\tilde{]}, \, \widetilde{A}) \in {\mathcal N}_T$
infinity is a singular critical point of the J-self-adjoint, J-non-negative and boundedly invertible operator $\widetilde{A}$.
For the exceptional element $(K_0, \, [\cdot,\cdot]_0, \, A_0) \in {\mathcal N}_T$ infinity is not a singular critical point of $A_0$.
This element is the regularization of 
$(K, \, [\cdot,\cdot], \, A) \, (\in {\mathcal N}_T)$
and it is characterized by the property $K_0 = \dom |T|^{\frac{1}{2}}$. 
\end{cor}
\begin{proof}
In view of Theorem \ref{thm_exceptual}
it remains to show (\ref{common_subset}) and the density of $\dom T^2$ in $(\widetilde{K}, \, [\cdot,\cdot\tilde{]})$.
Indeed, if $f \in \dom T^2$ then $Tf \in \dom T \subset \widetilde{K}$ and hence, $f \in \dom \widetilde{A} \, (= T^{-1}(\widetilde{K}))$.
Furthermore, for the dense subspace $\dom \widetilde{A}^2$ of $(\widetilde{K}, \, [\cdot,\cdot\tilde{]})$ 
we have $\dom \widetilde{A}^2 \subset \dom T^2$.
\end{proof}
Note that by (\ref{common_subset}) the regularization of a J-self-adjoint, J-non-negative and boundedly invertible operator in a Kre\u{\i}n space
can only be a ``small'' modification of the space and of the operator.
In particular, the  eigenvalues and eigenfunctions 
of the operators coincide.

Next, the eigenspectral functions associated with all operators from ${\mathcal N}_T$
are determined by the spectral measure associated with $T$.
\begin{thm}\label{thm_spectralFunction}
Let $T$ be a self-adjoint and boundedly invertible operator in the Hilbert space $(H, \, (\cdot,\cdot))$
and let $E$ denote the corresponding spectral measure allowing the representation
\[
T = \int_{-\infty}^\infty \lambda \, dE(\lambda).
\]
Then, for each element $(K, \, [\cdot,\cdot], \, A) \in {\mathcal N}_T$  the restrictions
\begin{equation}\label{eigenspectral_restriction}
E_K(\Delta) := E(\Delta)|_{K} \quad (\Delta \in \Sigma)
\end{equation}
define the eigenspectral function $E_K$ of $A$ in $(K, \, [\cdot,\cdot])$ according to (\ref{EigenspectralFunction}).
\end{thm}
\begin{proof}
Let $\widetilde{E}_K$ denote the eigenspectral function of 
the J-self-adjoint, J-non-negative and boundedly invertible operator $A$ 
according to (\ref{EigenspectralFunction}).
For $f \in K$ the limits in (\ref{EigenspectralFunction})
must be understood with respect to the topology in $(K, \, [\cdot,\cdot])$
(including the integral itself).
By the continuous embedding this also implies convergence
in $(H, \, (\cdot,\cdot))$ and additionally, we have
\[
(A - \lambda)^{-1}f = (T - \lambda)^{-1}f
\quad \mbox{ for all } 
\lambda \in {\mathbb C} \setminus {\mathbb R}, \, f \in K
\]
since $A$ is the range restriction of $T$.
Then, we can conclude $\widetilde{E}_K(\Delta) f = E(\Delta)f$
for a bounded interval $\emptyset \neq \Delta \in \Sigma$
by Stone's formula
for the spectral measure 
of a self-adjoint Hilbert space operator (see e.g. \cite[Problem VI 5.7]{Kato} or \cite[(1.5.4)]{BHS}).
This implies $\widetilde{E}_K(\Delta) = E(\Delta)|_{K}$ for all $\Delta \in \Sigma$
and hence, $\widetilde{E}_K = E_{K}$.
\end{proof}
Of course, Theorem \ref{thm_spectralFunction} also applies to the approach from Corollary \ref{cor_regularization}:
\begin{cor}\label{cor_spectral_function}
Let $A$ be a J-self-adjoint, J-non-negative and boundedly invertible operator in a Kre\u{\i}n space $(K, \, [\cdot,\cdot])$.
Let the Hilbert space $(K_-, \, \{\cdot , \cdot\}_-)$ and the self-adjoint operator $A_-$ be given according to Section \ref{sec_triplet}.
Furthermore, denote by $E$ the spectral measure of $A_-$.
Then, the eigenspectral function $E_K$ of $A$ satisfies (\ref{eigenspectral_restriction}).
\end{cor}
Using again the terminology of Section \ref{sec_triplet} and Corollary \ref{cor_spectral_function} \'Curgus' result from \cite[Proposition 3.1]{C2}
has the form
\[
f \in K_+ \quad \Leftrightarrow
\quad \int_{-\infty}^{\infty} \lambda \; d[E_K(\lambda)f,f] < \infty .
\]
Corollary \ref{cor_spectral_function} allows the following slight improvement 
(which however, can also be obtained directly from the proof in \cite[Proposition 3.1]{C2}).
\begin{cor}\label{cor_improvementBranko}
In the setting of Corollary \ref{cor_spectral_function} we have for $f \in K_+$
\[
\{f,f\}_+ = \int_{-\infty}^{\infty} \lambda \; d[E_K(\lambda)f,f] 
\]
\end{cor}
\begin{proof}
By Corollary \ref{cor_spectral_function} $E$ is an extension of $E_K$. Therefore we can calculate
\[
\int_{-\infty}^{\infty} \lambda \; d[E_K(\lambda)f,f] 
= \int_{-\infty}^{\infty} \lambda \; d\{E(\lambda)f, A_- f\}_-
= \{A_- f, A_- f \}_- 
= \{f, f \}_+ 
\]
since $A_- : K_+ \rightarrow K_-$ is isometric.
\end{proof}
Finally, it is mentioned that the results in this Section are related to 
the study of Kre\u{\i}n space completions by \'Curgus and Langer in \cite{CL2}.
In the terminology of \cite{CL2} the exceptional element $(K_0, \, [\cdot,\cdot]_0, \, A_0) \in {\mathcal N}_T$
is related to the so-called ``canonical'' Kre\u{\i}n space completion of $(\dom T, \, (T\cdot,\cdot))$.
Furthermore, translating \cite[Theorem 2.7]{CL2} into the present setting we obtain the following result:
\begin{thm}\label{thm_unique}
Let $T$ be a self-adjoint and boundedly invertible operator in the Hilbert space $(H, \, (\cdot,\cdot))$.
Then, the form $(T\cdot,\cdot)$ has a unique closure
(i.e. the sets ${\mathcal M}_T$ and ${\mathcal N}_T$ from Theorem \ref{thm_1-1_only_T}
both have only one element)
if and only if the operator $T$ is semibounded.
\end{thm}
\begin{proof}
By the assumption, the operator $S := T^{-1}$ is self-adjoint and bounded in $(H, \, (\cdot,\cdot))$.
The operator $T$ is semibounded if and only if $(c, \infty)$ or $(-\infty, -c)$
belongs to the resolvent set $\rho(T)$ for some $c > 0$.
With $\varepsilon := \frac{1}{c}$ this is further equivalent to
$(0,\varepsilon) \subset \rho(S)$ or $(-\varepsilon,0) \subset \rho(S)$.
Note that the inner product $(\cdot,\cdot)_S := (T\cdot,\cdot)$ on $\dom T$ (i.e. the range of $S$) satisfies
\[
(f,g)_S = (f,Tg) = (Sx, y) \quad \mbox{ if } f = Sx, \, g = Sy, \; x, y \in H
\]
and hence, $(\cdot,\cdot)_S$ coincides with the inner product defined in \cite[(2.2)]{CL2}.
Therefore, \cite[Theorem 2.7]{CL2} can be applied to $S$ and $(\cdot,\cdot)_S$
and it turns out that
$(\dom T, \, (\cdot,\cdot)_S)$ has a unique Kre\u{\i}n space completion 
which is continuously embedded in $(H, \, (\cdot,\cdot))$
if and only if $(0,\varepsilon) \subset \rho(S)$ or $(-\varepsilon,0) \subset \rho(S)$ for some $\varepsilon > 0$.
This means that the set ${\mathcal N}_T$ has precisely one element
if and only if $T$ is semibounded.
Then, the statement follows by means of Theorem \ref{thm_1-1_only_T}.
\end{proof}
By Theorem \ref{thm_unique} (or by classical arguments)
in the semibounded case each of the sets ${\mathcal M}_T$ and ${\mathcal N}_T$ from Theorem \ref{thm_1-1_only_T}
has only one element, namely the exceptional one.
From \cite[Theorem 5.2, proof part II]{CL2} one can conclude that in the non-semibounded case
the sets ${\mathcal M}_T$ and ${\mathcal N}_T$ have even infinitely many elements.
However, it is generally a non-trivial task to identify at least two different elements in these sets explicitly.
In the (two different) proofs of \cite[Theorem 5.2]{CL2} such constructions are presented on a certain abstract level.
Here, we try an approach to explicit examples.
First, we recall an example from \cite{F8}.
\begin{ex}\label{ex_SturmLiouville}
In the usual Hilbert space $L^2[-1,1]$ with inner product
$
(u,v) := \int_{-1}^1 u \overline{v} \, dx
$
consider the form $\st [\cdot,\cdot]$ given by 
\begin{eqnarray}
\dom \st & := & \{ u \in AC[-1,1] \, | \, \int_{-1}^1 |p| |u'|^2 \, dx < \infty, \,  u(-1)=u(1)=0  \},  \nonumber \\
\st [u,v] & := & \int_{-1}^1  u'\overline{v}' p \, dx
\quad (u, v \in \dom \st)     \nonumber
\end{eqnarray}
where the indefinite weight function $p$ is defined by
\begin{equation}\label{ex_singular_p}
p(x) := \left\{
\begin{array}{ll}
-\frac{2}{e}(1 - \log 2)^3  & \quad (x \in [-1,-\frac{2}{e}]) \\
 & \\
x \left|\log |x|\right|^3  & \quad (x \in (-\frac{2}{e},\frac{1}{e}) \setminus \{0\}) \\
 & \\
0  & \quad (x = 0) \\
 & \\
\frac{1}{e}  & \quad (x \in [\frac{1}{e},1])\\
\end{array} \right . \nonumber
\end{equation}
with Euler's constant $e\approx 2,718...$
In \cite[Section 2, Theorem 6.4]{F8} it is shown that $\st [\cdot,\cdot]$ is closed with gap point $0$
and for the function
\begin{equation}\label{ex_u}
u_0(x) := 
\left\{
\begin{array}{ll}
0  & \quad (x \in [-1,0]) \\
 & \\
\frac{8}{9 \, |\log x|^\frac{9}{8}}  & \quad (x \in (0,\frac{1}{e}) \\
 & \\
\frac{8}{9} - \frac{8(ex - 1)}{9(e-1)}  & \quad (x \in [\frac{1}{e},1])\\
\end{array} \right . 
\end{equation}
we have 
$u_0 \in \dom \st \setminus \dom |T|^\frac{1}{2}$.
Here, $T := T_{\st}$ is the associated operator in $L^2[-1,1]$ given by 
$Tu = -(pu')'$ defined on
\begin{eqnarray}
\dom T =  \{ u \in L^2[-1,1] & | &  u, pu' \in AC[-1,1], \, (pu')' \in L^2[-1,1], \,
\nonumber \\ 
 &  &
 u(-1)=u(1)=0  \}.
   \nonumber 
\end{eqnarray}
Consequently, the form $\st [\cdot,\cdot]$ is not regular
and hence, we can identify at least two elements in ${\mathcal M}_T$ in this example:
$\st [\cdot,\cdot]$ and its regularization.
However, here we cannot give an explict description of this regularization of $\st [\cdot,\cdot]$.

Furthermore, it is observed in \cite{F8} that the spectrum of $T$ is discrete and consists of a sequence of simple
eigenvalues
\begin{equation}
-\infty < ... < \lambda_{-2} < \lambda_{-1} <  0 
< \lambda_{1} < \lambda_{2} < ... < \infty            \nonumber
\end{equation}
accumulating only at $\infty$ and $-\infty$.
Let $u_n$ denote the corresponding eigenfunctions normed by
\begin{equation}
1 = |\lambda_n| (u_n,u_n) = (|T|^\frac{1}{2} u_n, |T|^\frac{1}{2}u_n) \quad (n \in {\mathbb Z} \setminus \{0\})    \nonumber
\end{equation}
and let $A_{\st}$ be the range restriction of $T$ to the Kre\u{\i}n space $(\dom \st, \, \st [\cdot,\cdot])$.
Then, for the eigenspectral function $E_{\dom \st}$ of $A_{\st}$ we have 
\begin{equation}\label{expansion}
E_{\dom \st}([-m,m])u = \Sigma_{|\lambda_n| \leq m, n \neq 0} \; \sgn(\lambda_n) \st[u,u_n] u_n
\end{equation}
according to \cite[Section 7]{F8}. 
Since the form $\st [\cdot,\cdot]$ is not regular infinity is a singular critical point of $A_{\st}$ by Theorem \ref{second}.
Indeed, by \cite[Theorem 6.4]{F8} for the function $u_0$ from (\ref{ex_u}) the ``eigenfunction expansion''
\begin{equation}\label{expansion_series}
\Sigma_{n \in {\mathbb Z} \setminus \{0\}} \; \sgn(\lambda_n) \st[u_0,u_n] u_n
\end{equation}
does not converge {\it unconditionally} in $(\dom \st, \, \st [\cdot,\cdot])$. 
In particular, by \cite[Theorem 6.4, Theorem 3.3]{F8} at least one of the two series 
\begin{eqnarray}
&& \Sigma_{n = 1}^{\infty} \; \st[u_0,u_n] u_n \, (= \lim_{m \rightarrow \infty} E_{\dom \st}([0,m])u_0 ),           \nonumber \\ 
&& \Sigma_{n = 1}^{\infty} \; \st[u_0,u_{-n}] u_{-n} \, (= \lim_{m \rightarrow \infty} - E_{\dom \st}([-m,0])u_0 )        \nonumber
\end{eqnarray}
does not converge in $(\dom \st, \, \st [\cdot,\cdot])$.
However, it remains as an open question in \cite{F8}
whether the rearrangement of the series (\ref{expansion_series}) according to (\ref{expansion}) does converge,
i.e. whether the limit $\lim_{m \rightarrow \infty} E_{\dom \st}([-m,m])u_0$ exists.
Below, we shall study a similar question in a different setting.
\end{ex}

\section{Examples with multiplication operators in ``model spaces''}\label{sec_example}

\subsection{A regular closed form}\label{subsec_example_regular}

We return to the setting of Examle \ref{ex_model_spaces} and change to the notations according to (\ref{rule_Phi_inverse}).
In this notation we consider the Hilbert space $(H, \, (\cdot,\cdot))$ given by
\[
H := K_- = L^2_{r_-}({\mathbb R}),
\qquad
(f,g) := \{f,g\}_- = (f,g)_{r_-} = \int_{-\infty}^\infty f \overline{g} \, r_- \, dx
\]
(according to (\ref{triplet_r})) and the form $\st [\cdot,\cdot]$ given by
\begin{equation}\label{form_t}
\dom \st := K = L^2_{r}({\mathbb R}),
\qquad
\st [f,g] := [f,g] = [f,g]_{r} = \int_{-\infty}^\infty f \overline{g} \, r \, dx 
.
\end{equation}
This form is closed with gap point $0$ and by Lemma \ref{lem_A_-}(iv)
the associated self-adjoint Hilbert space operator $T := T_{\st}$ in $H = L^2_{r_-}({\mathbb R})$ is given by
\begin{equation}\label{op_T}
\dom T = \dom A_- = L^2_{r_+}({\mathbb R}), \qquad (Tf)(x) = (A_-f)(x) = x f(x). 
\end{equation}
Here, we can explicitly describe the square root operator $|T|^{\frac{1}{2}}$ by
\[
\dom |T|^{\frac{1}{2}} = L^2_{r}({\mathbb R}), \qquad
(|T|^{\frac{1}{2}}f)(x) = \sqrt{|x|} f(x). 
\]
Therefore, we have $\dom \st = \dom |T|^{\frac{1}{2}}$
and hence, in this situation the form $\st [\cdot,\cdot]$ is a regular closed form,
i.e. a regular closure of $(T\cdot,\cdot)$.
This means that for the associated 
J-self-adjoint, J-non-negative and boundedly invertible range restriction $A_{\st} = A$
in the Kre\u{\i}n space $(L^2_{r}({\mathbb R}), \, [\cdot,\cdot]_{r})$
from Examle \ref{ex_model_spaces} infinity is not a singular critical point.
This is also evident in an explicit way:

Indeed, it is well known that the spectral measure $E$ of the multiplication operator $T$
is given by the multiplication operators
\begin{equation}\label{spectralMeasure}
E(\Delta) f := \chi_{\Delta} f \qquad (f \in L^2_{r_-}({\mathbb R}))
\end{equation}
for each measurable set $\Delta \subset {\mathbb R}$.
Here, $\chi_{\Delta}$ denotes the characteristic function of the set $\Delta$.
Then, by Theorem \ref{thm_spectralFunction}
the eigenspectral function $E_{\dom \st}$ of $A_{\st}$ 
is given by the restriction of the multiplication operator $E(\Delta)$
to $\dom \st = L^2_{r}({\mathbb R})$ for each $\Delta \in \Sigma$.
Therefore, for each $f \in L^2_{r}({\mathbb R})$ the limits
\begin{eqnarray}
&& \lim_{\lambda \rightarrow \infty} E_{\dom \st} ([\varepsilon,\lambda]) f 
= \lim_{\lambda \rightarrow \infty} \chi_{[\varepsilon,\lambda]} f = \chi_{[\varepsilon,\infty)} f,           \nonumber \\ 
&& \lim_{\lambda \rightarrow \infty} E_{\dom \st} ([-\lambda,\varepsilon]) f 
= \lim_{\lambda \rightarrow \infty} \chi_{[-\lambda,\varepsilon]} f = \chi_{(-\infty,\varepsilon]} f        \nonumber
\end{eqnarray}
exist in the Kre\u{\i}n space $(L^2_{r}({\mathbb R}), \, [\cdot,\cdot]_r)$
(using $\varepsilon > 0$ from (\ref{function_r_sign})).
Again, this shows that infinity is not a singular critical point of $A_{\st}$ .

\subsection{Non-regular closed forms}\label{subsec_example_nonregular}

We proceed with the setting of Section \ref{subsec_example_regular}.
However, we now identify an infinite number of non-regular closures 
(or Kre\u{\i}n space completions)
of the form $(T\cdot,\cdot)$ defined on $\dom T = L^2_{r_+}({\mathbb R})$.
This construction is inspired by \cite[Theorem 5.2, proof part II]{CL2}.
To this end, for $\alpha \in [0,2]$ we first introduce the even functions
\[
\eta_{\alpha}(x) := (\sqrt{|x|^{\alpha} + 1} - \sqrt{|x|^{\alpha}})|r(x)|, \quad
\omega_{\alpha}(x) := \sqrt{|x|^{\alpha}} \, |r(x)| \quad (x \in {\mathbb R}).
\]
Furthermore, for a complex function $f$ on ${\mathbb R}$ we define the function
\begin{equation}\label{Qalpha}
({\mathcal Q}_{\alpha}f)(x) := \sqrt{|x|^{\alpha} + 1 } \, f(x) - \sqrt{|x|^{\alpha}} \, f(-x)  \quad (x \in {\mathbb R})
\end{equation}
and the even and odd part
\[
f_e(x) := \frac{1}{2}(f(x) + f(-x)), \quad f_o(x) := \frac{1}{2}(f(x) - f(-x))  \quad (x \in {\mathbb R}).
\]
Since $\eta_{\alpha}$ and $\omega_{\alpha}$ are non-negative functions on ${\mathbb R}$
we can consider the associated weighted Hilbert spaces 
$ L^2_{\eta_{\alpha}}({\mathbb R})$ and $ L^2_{\omega_{\alpha}}({\mathbb R})$
equipped with the inner products
\begin{eqnarray}
(f,g)_{\eta_{\alpha}} := \int_{-\infty}^\infty f \overline{g} \, \eta_{\alpha} \, dx && \quad (f,g \in  L^2_{\eta_{\alpha}}({\mathbb R})),    \nonumber \\ 
(f,g)_{\omega_{\alpha}} :=  \int_{-\infty}^\infty f \overline{g} \, \omega_{\alpha} \, dx &&  \quad (f,g \in  L^2_{\omega_{\alpha}}({\mathbb R})).     \nonumber
\end{eqnarray}
In a first step these spaces will be used in order to construct  a Hilbert space structure on the space
\begin{equation}\label{dom_t_alpha}
\dom \st_{\alpha} := \{f \in  L^2_{r_-}({\mathbb R}) \, | \, f_e \in  L^2_{\eta_{\alpha}}({\mathbb R}), \, f_o \in  L^2_{\omega_{\alpha}}({\mathbb R}) \} .
\end{equation}
To this end, we first identify the ``extremal'' cases $\alpha = 0$ and $\alpha = 2$ and observe that $\eta_{\alpha}$ behaves like
\[
\widetilde{\eta}_{\alpha}(x) := \frac{|r(x)|}{\sqrt{|x|^{\alpha}}} \quad (x \in {\mathbb R} \setminus \{0\})
\]
for $\alpha \in [0,2]$.
Note that by (\ref{function_r_sign}) the functions $\omega_{\alpha}, \, \eta_{\alpha}, \, \widetilde{\eta}_{\alpha}$ vanish on $[-\varepsilon,\varepsilon]$.
\begin{lem}\label{lem_extremal}
The following statements hold true
\begin{enumerate}
\item[(i)] for $\alpha \in (0,2]$: $\quad \frac{\eta_{\alpha}(x)}{\widetilde{\eta}_{\alpha}(x)} \longrightarrow \frac{1}{2}  \quad (|x| \longrightarrow \infty)$ ,
\item[(ii)] for $\alpha = 0$: $\quad \eta_{0}(x) = (\sqrt{2} - 1) \widetilde{\eta}_{0}(x) = (\sqrt{2} - 1) |r(x)| , \quad \omega_0(x) = |r(x)|$,
\item[(iii)] for $\alpha \in [0,2]$: $\quad L^2_{\eta_{\alpha}}({\mathbb R}) = L^2_{\widetilde{\eta}_{\alpha}}({\mathbb R})$,
\item[(iv)] for $\alpha = 0$: $\quad L^2_{\eta_{0}}({\mathbb R}) = L^2_{\omega_{0}}({\mathbb R}) = \dom \st_{0} = L^2_{r}({\mathbb R}) = \dom \st$,
\item[(v)] for $\alpha = 2$: $\quad \widetilde{\eta}_{2}(x) = r_-(x) , \quad \omega_2(x) = r_+(x)$,
\item[(vi)] for $\alpha = 2$: $\quad L^2_{\eta_{2}}({\mathbb R}) = L^2_{r_-}({\mathbb R}), \quad L^2_{\omega_{2}}({\mathbb R}) = L^2_{r_+}({\mathbb R})$.
\end{enumerate}
\end{lem}
\begin{proof}
By the mean value theorem (applied to $y(t) := \sqrt{t}$) we know that
\[
\sqrt{|x|^{\alpha} + 1} - \sqrt{|x|^{\alpha}} = \frac{1}{2 \sqrt{\xi}}
\]
with some $\xi \in [|x|^{\alpha}, |x|^{\alpha} +1]$.
Therefore, for all $x > \epsilon$ we can estimate
\[
\frac{\eta_{\alpha}(x)}{\widetilde{\eta}_{\alpha}(x)} = \sqrt{|x|^{\alpha}} (\sqrt{|x|^{\alpha} + 1} - \sqrt{|x|^{\alpha}}) \leq \frac{1}{2}
\]
and
\[
\frac{\eta_{\alpha}(x)}{\widetilde{\eta}_{\alpha}(x)} \geq \frac{\sqrt{|x|^{\alpha}}}{2 \sqrt{|x|^{\alpha} + 1}}
= \frac{1}{2 \sqrt{1 + |x|^{-\alpha}}}
\longrightarrow \frac{1}{2}  \quad (x \longrightarrow \infty)
\]
for $\alpha \neq 0$. This implies (i) and (iii) follows immediately.
The other statements are obvious.
\end{proof}
Next, we study the relations between these spaces for $\alpha \in [0,2]$.
\begin{lem}\label{lem_int_exist}
\begin{enumerate}
\item[(i)] We have
\[
L^2_{r_+}({\mathbb R}) \subset L^2_{\omega_{\alpha}}({\mathbb R}) \subset L^2_{\eta_{\alpha}}({\mathbb R}) \subset L^2_{r_-}({\mathbb R}).
\]
\item[(ii)] There is a constant $c > 0$ such that 
\[
(f,f)_{\omega_{\alpha}} \leq c (f,f)_{r_+}, \quad (g,g)_{\eta_{\alpha}} \leq c (g,g)_{\omega_{\alpha}}, \quad (h,h)_{r_-} \leq c (h,h)_{\eta_{\alpha}}
\]
for all $f \in L^2_{r_+}({\mathbb R}), \; g \in L^2_{\omega_{\alpha}}({\mathbb R}), \; h \in L^2_{\eta_{\alpha}}({\mathbb R})$.
\item[(iii)] For $f \in L^2_{\omega_{\alpha}}({\mathbb R})$ also the function $f(-x)$ belongs to $L^2_{\omega_{\alpha}}({\mathbb R})$ and we have
\begin{eqnarray}
\int_{-\infty}^\infty |f(-x)|^2 \, \omega_{\alpha}(x) \, dx & = & \int_{-\infty}^\infty |f(x)|^2 \, \omega_{\alpha}(x)  \, dx ,    \nonumber \\ 
\int_{-\infty}^\infty f(-x) \overline{f(x)} \, \omega_{\alpha}(x)  \, dx & = & \int_{-\infty}^\infty f(x) \overline{f(-x)} \, \omega_{\alpha}(x)  \, dx.     \nonumber
\end{eqnarray}
\item[(iv)] The statement of (iii) remains true with $\omega_{\alpha}$ replaced by $\eta_{\alpha}$.
\item[(v)] For $f, g \in L^2_{\omega_{\alpha}}({\mathbb R})$ the following intergal exists and we have
\begin{equation}\label{pos_inner_prod}
\int_{-\infty}^\infty ({\mathcal Q}_{\alpha}f) \overline{g} \, |r| \, dx
= 2 (f_o,g_o)_{\omega_{\alpha}} + (f,g)_{\eta_{\alpha}} .
\end{equation}
\end{enumerate}
\end{lem}
\begin{proof}
(i), (ii):
For $\alpha \in (0,2)$
we obtain the convergence results
\[
\frac{\omega_{\alpha}(x)}{r_+(x)} = |x|^{\frac{\alpha - 2}{2}}  \searrow 0, \quad
\frac{\eta_{\alpha}(x)}{\omega_{\alpha}(x)} = \sqrt{1 + |x|^{-\alpha} } - 1 \searrow 0 \quad (x \longrightarrow \infty)
\]
and by Lemma \ref{lem_extremal}(i)
\[
\frac{r_-(x)}{\eta_{\alpha}(x)} 
\leq d \frac{r_-(x)}{\widetilde{\eta}_{\alpha}(x)} 
= d \frac{\sqrt{|x|^{\alpha}}}{|x|}
= d |x|^{\frac{\alpha - 2}{2}}  \searrow 0 \quad (x \longrightarrow \infty)
\]
with some $d > 0$.
Therefore, by (\ref{function_r_sign}) these fractions are bounded
and for $\alpha \in \{0,2\}$ this is also true by similar calculations.
This implies (i) und (ii) with a common upper bound $c > 0$. Indeed, e.g. for $f \in L^2_{r_+}({\mathbb R})$ we have
\begin{eqnarray}
& & \int_{-\infty}^\infty |f(x)|^2 \, \omega_{\alpha}(x)  \, dx        \nonumber \\
& = & \int_{-\infty}^{-\varepsilon} |f(x)|^2 \, \frac{\omega_{\alpha}(x)}{r_+(x)} r_+(x)  \, dx
 + \int_{\varepsilon}^\infty |f(x)|^2 \, \frac{\omega_{\alpha}(x)}{r_+(x)} r_+(x)  \, dx        \nonumber \\
& \leq & c \int_{-\infty}^\infty |f(x)|^2 \,  r_+(x)  \, dx .   \nonumber
\end{eqnarray}

(iii) For $f \in L^2_{\omega_{\alpha}}({\mathbb R})$ we have
\[
\int_{-\infty}^\infty |f(-x)|^2 \, \omega_{\alpha}(x)  \, dx 
= \int_{-\infty}^\infty |f(t)|^2 \, \omega_{\alpha}(-t)  \, dt 
= \int_{-\infty}^\infty |f(t)|^2 \, \omega_{\alpha}(t)  \, dt .
\]
since $\omega_{\alpha}$ is an even function.
The second equation in (iii) follows similarly.

(iv) can be shown by the same arguments as in (iii) since also $\eta_{\alpha}$ is even.

(v) First, for $g = f$ we calculate
\begin{eqnarray}
& & \int_{-\infty}^\infty \left(\sqrt{|x|^{\alpha} + 1 } f(x) - \sqrt{|x|^{\alpha}} \, f(-x)\right) \overline{f(x)} \, |r(x)| \, dx  
\nonumber \\
& = & \int_{-\infty}^\infty \left(\eta_{\alpha}(x)f(x) + \omega_{\alpha}(x)f(x) - \omega_{\alpha}(x)f(-x) \right)  \overline{f(x)} \, dx  \nonumber \\
& = & \int_{-\infty}^\infty \left(|f(x)|^2  - f(-x)\overline{f(x)} \right)  \omega_{\alpha}(x) \, dx
+ \int_{-\infty}^\infty |f(x)|^2 \eta_{\alpha}(x) \, dx   .      \nonumber
\end{eqnarray}
Writing $|f(x)|^2  - f(-x)\overline{f(x)}$ as $\frac{1}{2}(|f(x)|^2  - 2 f(-x)\overline{f(x)} + |f(x)|^2)$
and using (iii)
we end up with
\begin{eqnarray}
& & \int_{-\infty}^\infty ({\mathcal Q}_{\alpha}f) \overline{f} \, |r| \, dx 
\nonumber \\
& = & \frac{1}{2} \int_{-\infty}^\infty (|f(x)|^2  - f(-x) \overline{f(x)} - f(x) \overline{f(-x)} + |f(-x)|^2) \, \omega_{\alpha}(x)  \, dx   
\nonumber \\
&   & \qquad  
+  \int_{-\infty}^\infty |f(x)|^2 \, \eta_{\alpha}(x)  \, dx  \nonumber \\
& =  &  \frac{1}{2} \int_{-\infty}^\infty |f(x) - f(-x)|^2 \, \omega_{\alpha}(x)  \, dx +  \int_{-\infty}^\infty |f(x)|^2 \, \eta_{\alpha}(x)  \, dx  \nonumber \\
& = &   2 (f_o,f_o)_{\omega_{\alpha}} + (f,f)_{\eta_{\alpha}} .     \nonumber 
\end{eqnarray}
Finally, the statement (v) follows by the polarization identity.
Note that this conclusion is allowed since obviously,
$2 (f_o,g_o)_{\omega_{\alpha}} + (f,g)_{\eta_{\alpha}}$
defines a symmetric sesquilinear form on $L^2_{\omega_{\alpha}}({\mathbb R})$
and this is also true for $\int_{-\infty}^\infty ({\mathcal Q}_{\alpha}f) \overline{g} \, |r| \, dx$.
Indeed, this follows as in (iii) by the calculation
\[
\int_{-\infty}^\infty \sqrt{|x|^{\alpha}} \, f(-x)\overline{g}(x) |r(x)|  \, dx = \int_{-\infty}^\infty \sqrt{|t|^{\alpha}} \, f(t)\overline{g}(-t) |r(t)|  \, dx 
\]
for $f,g \in L^2_{\omega_{\alpha}}({\mathbb R})$.
\end{proof}
By means of (\ref{pos_inner_prod}) we can now define a Hilbert space structure on $\dom \st_{\alpha}$:
\begin{prop}\label{prop_Hilbert_space}
Let $\alpha \in [0,2]$. Then, the following statements hold true:
\begin{enumerate}
\item[(i)]  $\dom \st_{\alpha}$ is a Hilbert space with the inner product
\begin{equation}\label{t_alpha_pos}
\st_{\alpha}(f,g) := 2 (f_o,g_o)_{\omega_{\alpha}} + (f,g)_{\eta_{\alpha}} \quad (f, g \in \dom \st_{\alpha}).
\end{equation}
\item[(ii)]  We have the space triplet
\[
L^2_{r_+}({\mathbb R}) \subset \dom \st_{\alpha} \subset L^2_{r_-}({\mathbb R}).
\]
\item[(iii)]  Each of the inclusions in (ii) is continuous with respect to the corresponding Hilbert space inner products, i.e.
the inclusion of $(L^2_{r_+}({\mathbb R}), \, (\cdot,\cdot)_{r_+})$ in $(\dom \st_{\alpha}, \, \st_{\alpha} (\cdot,\cdot))$ and of this space in $(L^2_{r_-}({\mathbb R}), \, (\cdot,\cdot)_{r_-})$.
\end{enumerate}
\end{prop}
\begin{proof}
(i) First, we observe that 
$L_o := \{ f \in L^2_{\omega_{\alpha}}({\mathbb R}) \, | \, f $ is odd $ \}$ is a Hilbert space with $(\cdot,\cdot)_{\omega_{\alpha}}$.
Indeed, $L_o$ is the kernel of $I + S$ where $S$ is the self-adjoint and 
bounded operator $(Sf)(x) := f(-x)$ 
in $(L^2_{\omega_{\alpha}}({\mathbb R}), \, (\cdot,\cdot)_{\omega_{\alpha}})$
(cf. Lemma \ref{lem_int_exist} (iii)).
Similarly, also 
$L_e := \{ f \in L^2_{\eta_{\alpha}}({\mathbb R}) \, | \, f $ is even $ \}$ is a Hilbert space with $(\cdot,\cdot)_{\eta_{\alpha}}$
(cf. Lemma \ref{lem_int_exist} (iv)).
Therefore, the orthogonal sum $\dom \st_{\alpha} = L_e \oplus L_o$ is also a Hilbert space with the corresponding inner product
$(f_e,g_e)_{\eta_{\alpha}} + (f_o,g_o)_{\omega_{\alpha}}$
defined for $f_e, g_e \in L_e$ and $f_o, g_o \in L_o$.
However, this inner product is equivalent to $\st_{\alpha} (\cdot,\cdot)$.
Indeed, for all $f = f_e + f_o \in \dom \st_{\alpha}$
we have
\[
(f_e,f_e)_{\eta_{\alpha}} \leq (f_e,f_e)_{\eta_{\alpha}} + (f_o,f_o)_{\eta_{\alpha}} = (f,f)_{\eta_{\alpha}}
\]
by the orthogonality of $f_e$ and $f_o$ with respect to $(\cdot,\cdot)_{\eta_{\alpha}}$ 
and similarly, using Lemma \ref{lem_int_exist} (ii) we have
\[
(f,f)_{\eta_{\alpha}} = (f_e,f_e)_{\eta_{\alpha}} + (f_o,f_o)_{\eta_{\alpha}} \leq (f_e,f_e)_{\eta_{\alpha}} + c (f_o,f_o)_{\omega_{\alpha}}.
\]

(ii) and (iii) follow similarly from Lemma \ref{lem_int_exist}. First, for
$f = f_e + f_o \in L^2_{r_+}({\mathbb R})$
we have by Lemma \ref{lem_int_exist} (ii)
\[
\st_{\alpha} (f,f) 
\leq 2 c (f_o,f_o)_{r_+} + c^2 (f,f)_{r_+} \leq (2c + c^2) (f,f)_{r_+}
\]
since $f_e$ and $f_o$ are also othogonal with respect to $(\cdot,\cdot)_{r_+}$.
Furthermore, for
$f = f_e + f_o \in \dom \st_{\alpha}$
we have by Lemma \ref{lem_int_exist} (ii)
\[
(f,f)_{r_-}
\leq (f_o,f_o)_{r_-} + (f,f)_{r_-} \leq c^2 (f_o,f_o)_{\omega_{\alpha}} + c (f,f)_{\eta_{\alpha}} \leq d \, \st_{\alpha} (f,f) 
\]
with $d := \max \{c, \frac{c^2}{2}\}$.
\end{proof}
Now, we have a closer look at some set differences using the functions
\begin{equation}\label{f_g_tau}
f_{\tau}(x) := \frac{\sgn(x)}{\sqrt{|r(x)|} |x|^{\frac{\tau + 2}{4}}}, \quad
g_{\tau}(x) := \frac{1}{\sqrt{|r(x)|} |x|^{\frac{2 - \tau}{4}}}
\quad (x \in {\mathbb R} \setminus [-\varepsilon, \varepsilon])
\end{equation}
for $\tau \in [0,2]$ and $\varepsilon > 0$ according to (\ref{function_r_sign}).
Formally, these functions are extended by $f_{\tau}(x) := 0 =: g_{\tau}(x)$ for $x \in [-\varepsilon, \varepsilon]$.
In particular, $f_{\tau}$ is odd and $g_{\tau}$ is even.
\begin{lem}\label{lem_set_diff}
\begin{enumerate}
For $0 \leq \alpha < \beta \leq 2$ the following statements hold true:
\item[(i)] $L^2_{\omega_{\beta}}({\mathbb R}) \subset L^2_{\omega_{\alpha}}({\mathbb R}) \subset L^2_{r}({\mathbb R}$,
\item[(ii)] $f_{\beta} \in L^2_{\omega_{\alpha}}({\mathbb R}) \setminus L^2_{\omega_{\beta}}({\mathbb R})$,
\item[(iii)] $f_{\beta} \in \dom \st_{\alpha} \setminus \dom \st_{\beta}$,
\item[(iv)] $ L^2_{r}({\mathbb R} \subset L^2_{\eta_{\alpha}}({\mathbb R}) \subset L^2_{\eta_{\beta}}({\mathbb R})$,
\item[(v)] $g_{\alpha} \in L^2_{\eta_{\beta}}({\mathbb R}) \setminus L^2_{\eta_{\alpha}}({\mathbb R})$,
\item[(vi)] $g_{\alpha} \in \dom \st_{\beta} \setminus \dom \st_{\alpha}$.
\end{enumerate}
\end{lem}
\begin{proof}
(i)  For $\alpha = 0$ we have the identity $L^2_{\omega_{0}}({\mathbb R}) = L^2_{r}({\mathbb R})$. Therefore,
the inclusions in (i) follow like in Lemma \ref{lem_int_exist} (i) from the convergence
\[
\frac{\omega_{\alpha}(x)}{\omega_{\beta}(x)} = |x|^{\frac{\alpha - \beta}{2}}  \searrow 0 \quad (x \longrightarrow \infty).
\]

(ii) For $0 \leq \gamma \leq 2$ we have
\[
\int_{-\infty}^\infty |f_{\beta}(x)|^2 \, \omega_{\gamma}(x)  \, dx 
= \int_{-\infty}^{-\varepsilon} |x|^{\frac{\gamma - \beta - 2}{2}}  \, dx
+ \int_{\varepsilon}^\infty |x|^{\frac{\gamma - \beta - 2}{2}}  \, dx
\]
which is a finite number if and only if $\gamma < \beta$.

(iii) follows from (ii) since $f_{\beta}$ is odd.

(iv) For $\alpha = 0$ we have the identity $L^2_{\eta_{0}}({\mathbb R}) = L^2_{r}({\mathbb R})$ .
Therefore, the inclusions in (iv) follow like in Lemma \ref{lem_int_exist} (i) from the convergence
\[
\frac{\eta_{\beta}(x)}{\eta_{\alpha}(x)} \leq d \frac{\widetilde{\eta}_{\beta}(x)}{\widetilde{\eta}_{\alpha}(x)} = d |x|^{\frac{\alpha - \beta}{2}}  \searrow 0 \quad (x \longrightarrow \infty)
\]
with some $d > 0$ using Lemma \ref{lem_extremal}(i)(ii).

(v) For $0 \leq \gamma \leq 2$ we have
\[
\int_{-\infty}^\infty |g_{\alpha}(x)|^2 \, \widetilde{\eta}_{\gamma}(x)  \, dx 
= \int_{-\infty}^{-\varepsilon} |x|^{\frac{\alpha - 2 - \gamma}{2}}  \, dx
+ \int_{\varepsilon}^\infty |x|^{\frac{\alpha - 2 - \gamma}{2}}  \, dx
\]
which is a finite number if and only if $\gamma > \alpha$. Then, we use Lemma \ref{lem_extremal}(iii).

(vi) follows from (v) since $g_{\alpha}$ is even.
\end{proof}
In the next step we shall construct a Kre\u{\i}n space structure on the Hilbert space $(\dom \st_{\alpha}, \, \st_{\alpha} (\cdot,\cdot))$
by means of the representation (\ref{pos_inner_prod}) of $\st_{\alpha} (\cdot,\cdot)$.
To this end, we first 
introduce two operators on the linear space of weighted functions with compact support
\[
L_{r, 0}^2 = \{f \in L^2_{r}({\mathbb R}) \, | \, f(x) = 0 \mbox{ a.e. on } {\mathbb R} \setminus [-k, k] \mbox{ with some } k > 0 \}.
\]
Using the expression ${\mathcal Q}_{\alpha}f$ from (\ref{Qalpha})
the operators $Q_{\alpha}$ and $S_{\alpha}$ given by
\[
Q_{\alpha}f := {\mathcal Q}_{\alpha}f, \quad (S_{\alpha}f)(x) := \sgn (x) ({\mathcal Q}_{\alpha}f)(x) \quad (f \in L_{r, 0}^2)
\]
are well defined in $L_{r, 0}^2$
since for each $f \in L_{r, 0}^2$ also ${\mathcal Q}_{\alpha}f$ has a compact support.
Furthermore, we consider the kernels
\begin{equation}\label{sets_Mpm}
{\mathcal M}_{\alpha}^+ := \ker (I - S_{\alpha}), 
\quad {\mathcal M}_{\alpha}^- := \ker (I + S_{\alpha}) 
\end{equation}
where $I$ denotes the identity on $L_{r, 0}^2$.
Note that obviously, we have $L_{r, 0}^2 \subset L^2_{r_+}({\mathbb R})$
and hence, $L_{r, 0}^2 \subset L^2_{\omega_{\alpha}}({\mathbb R})$,
 $L_{r, 0}^2 \subset \dom \st_{\alpha}$
by Lemma \ref{lem_int_exist} and Proposition \ref{prop_Hilbert_space}, respectively.
Moreover, (\ref{pos_inner_prod}) and (\ref{t_alpha_pos}) allow the representation
\begin{equation}\label{reprQ}
\st_{\alpha} (f,g) = (Q_{\alpha}f, g)_r \qquad (f, g \in L_{r, 0}^2).
\end{equation}
\begin{lem}\label{lem_Mpm}
\begin{enumerate}
\item[(i)] The operator $Q_{\alpha}$ is symmetric with respect to $(\cdot,\cdot)_r$, i.e.
\[
(Q_{\alpha}f, g)_r = (f, Q_{\alpha}g)_r \quad 
(f,g \in L_{r, 0}^2)
\]
\item[(ii)] For $f \in L_{r, 0}^2$ we have $S_{\alpha}^2f = f$.
\item[(iii)] The space $L_{r, 0}^2$ allows the decomposition
\begin{equation}\label{decomp_domA+}
L_{r, 0}^2 = {\mathcal M}_{\alpha}^+ \oplus {\mathcal M}_{\alpha}^-.
\end{equation}
\item[(iv)] The spaces ${\mathcal M}_{\alpha}^+$ and ${\mathcal M}_{\alpha}^-$
are orthogonal with respect to $\st_{\alpha} (\cdot,\cdot)$, i.e. the decomposition (\ref{decomp_domA+})
is a direct and orthogonal sum in this sense.
\item[(v)] For $f_+ \in {\mathcal M}_{\alpha}^+, \, f_-  \in {\mathcal M}_{\alpha}^-$ we have
\[
\st_{\alpha} (f_+,f_+) = [f_+,f_+]_r, \quad \st_{\alpha} (f_-,f_-) = - [f_-,f_-]_r.
\]
\item[(vi)] The spaces ${\mathcal M}_{\alpha}^+$ and ${\mathcal M}_{\alpha}^-$
are also orthogonal with respect to $[\cdot,\cdot]_r$.
\item[(vii)] For $f \in \dom \st_{\alpha}, 
\, k \in {\mathbb N}$ 
we have $\chi_{[-k,k]} f \in L_{r, 0}^2$ and
\[
\chi_{[-k,k]} f \longrightarrow f 
\quad (k \longrightarrow \infty)
\quad \mbox{ w.r.t. } \st_{\alpha} (\cdot,\cdot).
\]
\item[(viii)] For $g \in L^2_{r_+}({\mathbb R}), 
\, k \in {\mathbb N}$ 
we have $\chi_{[-k,k]} g \in L_{r, 0}^2$ and
\[
\chi_{[-k,k]} g \longrightarrow g 
\quad (k \longrightarrow \infty)
\quad \mbox{ w.r.t. } (\cdot,\cdot)_{r_+}.
\]
\end{enumerate}
\end{lem}
\begin{proof}
(i) From Lemma \ref{lem_int_exist} (iii) (or
by an immediate calculation)
we obtain
\[
\int_{-\infty}^\infty \sqrt{|x|^{\alpha}} \, f(-x)\overline{g}(x) |r(x)|  \, dx = \int_{-\infty}^\infty \sqrt{|t|^{\alpha}} \, f(t)\overline{g}(-t) |r(t)|  \, dx 
\]
for $f,g \in L_{r, 0}^2$.
This implies (i) by the definition of ${\mathcal Q}_{\alpha}f$ in (\ref{Qalpha}).
(Of course, (i) can also be deduced from (\ref{reprQ}).)

(ii) Using $\sgn (-x) = -\sgn (x)$ we can calculate for $f \in L_{r, 0}^2$
\begin{eqnarray}
(S_{\alpha}^2f)(x) & = & \sgn (x) (\sqrt{|x|^{\alpha} + 1 } (S_{\alpha}f)(x) - \sqrt{|x|^{\alpha}} \, (S_{\alpha}f)(-x))   \nonumber \\
& = & \sqrt{|x|^{\alpha} + 1 } (\sqrt{|x|^{\alpha} + 1 } f(x) - \sqrt{|x|^{\alpha}} \, f(-x))       
\nonumber \\
& & 
+ \sqrt{|x|^{\alpha}} \, (\sqrt{|x|^{\alpha} + 1 } f(-x) - \sqrt{|x|^{\alpha}} \, f(x))   = f(x).    \nonumber
\end{eqnarray}

(iii) Each $f \in L_{r, 0}^2$ can be written as $f = f_+ + f_-$ with
\[
f_+ := \frac{1}{2}(f + S_{\alpha}f), \quad f_- := \frac{1}{2}(f - S_{\alpha}f) \in L_{r, 0}^2
\]
and we have
\[
(I - S_{\alpha})f_+ = \frac{1}{2}(f - S_{\alpha}^2f) = 0, 
\qquad (I + S_{\alpha})f_- = \frac{1}{2}(f - S_{\alpha}^2f) = 0.
\]
Therefore, $f_{\pm}$ belongs to ${\mathcal M}_{\alpha}^{\pm}$
and we have proved $L_{r, 0}^2 = {\mathcal M}_{\alpha}^+ + {\mathcal M}_{\alpha}^-$.

(iv) For $f_+ \in {\mathcal M}_{\alpha}^+, \, f_-  \in {\mathcal M}_{\alpha}^-$ 
we know that $S_{\alpha}f_+ = f_+$ and $S_{\alpha}f_- = -f_-$.
Consequently, we can first calculate by (\ref{reprQ})
\[
\st_{\alpha} (f_+,f_-) = (Q_{\alpha}f_+, f_-)_r =  (J S_{\alpha}f_+, f_-)_r =  (J f_+, f_-)_r = [f_+, f_-]_r 
\]
using the fundamental symmetry $J$ from (\ref{fundSymm}). Similarly, we obtain by (i)
\[
\st_{\alpha} (f_+,f_-) = (f_+, Q_{\alpha}f_-)_r =  (f_+, J S_{\alpha}f_-)_r =  - (f_+, J f_-)_r = - [f_+, f_-]_r.
\]
Therefore, we have $\st_{\alpha} (f_+,f_-) = 0$ which implies (iv).

(v) follows by a similar calculation as in (iv).

(vi) also follows from the calculation in (iv).

(vii) Put
$f_k := \chi_{[-k,k]} f$. 
Then, we have
$f_k \in L_{r, 0}^2$ and we can estimate
\begin{eqnarray}
& & \st_{\alpha} (f - f_k,f - f_k)  =  2 (f_o - {f_k}_o,f_o - {f_k}_o)_{\omega_{\alpha}} + (f - f_k,f - f_k)_{\eta_{\alpha}}   \nonumber \\
& = & \frac{1}{2} \int_{-\infty}^{-k} |f(x) - f(-x)|^2 \, \omega_{\alpha}(x)  \, dx  + \frac{1}{2} \int_{k}^{\infty} |f(x) - f(-x)|^2 \, \omega_{\alpha}(x)  \, dx   \nonumber  \\
&  & + \int_{-\infty}^{-k} |f(x)|^2 \, \eta_{\alpha}(x)  \, dx  +  \int_{k}^{\infty} |f(x)|^2 \, \eta_{\alpha}(x)  \, dx
\; \longrightarrow 0 \quad (k \longrightarrow \infty)   \nonumber 
\end{eqnarray}
where $f_o$ and ${f_k}_o$ denote the odd parts of $f$ and ${f_k}$, respectively.

(viii) Put
$g_k := \chi_{[-k,k]} g$. 
Then, we have
$g_k \in L_{r, 0}^2$ and we can estimate
\[
(g - g_k,g - g_k)_{r_+} = \int_{-\infty}^{-k} |g|^2 \, |r_+|  \, dx  + \int_{k}^{\infty} |g|^2 \, |r_+| \, dx   
\; \longrightarrow 0 \quad (k \longrightarrow \infty).   \nonumber 
\]
\end{proof}
An orthogonal decomposition is preserved when we go over to the closure
and by Lemma \ref{lem_Mpm}(vii) $L_{r, 0}^2$ is dense in the Hilbert space $(\dom \st_{\alpha}, \, \st_{\alpha} (\cdot,\cdot))$.
Therefore, (\ref{decomp_domA+}) implies
\begin{equation}\label{decomp_dom_t_alpha}
\dom \st_{\alpha} = \overline{L_{r, 0}^2} = \overline{{\mathcal M}_{\alpha}^+} \oplus \overline{{\mathcal M}_{\alpha}^-}
\end{equation}
with a direct and othogonal sum in $(\dom \st_{\alpha}, \, \st_{\alpha} (\cdot,\cdot))$
where $\overline{L_{r, 0}^2}, \overline{{\mathcal M}_{\alpha}^+}, \overline{{\mathcal M}_{\alpha}^-}$
denote the closures in this space.
If $P_{\alpha}^{\pm}$ denotes the orthogonal projection onto $\overline{{\mathcal M}_{\alpha}^{\pm}}$
in this space then, with
\[
J_{\alpha} := P_{\alpha}^+ - P_{\alpha}^-
\]
the sesquilinear form
\begin{equation}\label{KreinSpaceInnerProd}
\st_{\alpha}[f,g] := \st_{\alpha}(J_{\alpha}f,g) \qquad (f,g \in \dom \st_{\alpha})
\end{equation}
defines a Kre\u{\i}n space inner product on $\dom \st_{\alpha}$.
Furthermore, (\ref{decomp_dom_t_alpha}) is a corresponding fundamental decomposition,
$J_{\alpha}$ is a corresponding fundamental symmetry
and $\st_{\alpha} (\cdot,\cdot)$ is the corresponding Hilbert space inner product.
On $L_{r, 0}^2$ we can make this structure more explicit.
\begin{lem}\label{lem_t_alpha_on_L2r0}
For $f, g \in L_{r, 0}^2$ we have
\begin{eqnarray}
P_{\alpha}^+f = \frac{1}{2} (f + S_{\alpha}f),  & &   P_{\alpha}^-f = \frac{1}{2} (f - S_{\alpha}f), \quad J_{\alpha}f = S_{\alpha}f, \label{P_alpha_pm} \\
& & \st_{\alpha}[f,g] = [f,g]_r.   \label{t_alpha_on_Lr0}
\end{eqnarray}
\end{lem}
\begin{proof}
By Lemma \ref{lem_Mpm}(ii) we observe that
\[
(I - S_{\alpha})(f + S_{\alpha}f) = f - S_{\alpha}^2f = 0, \quad (I + S_{\alpha})(f - S_{\alpha}f) = f - S_{\alpha}^2f = 0
\]
and hence $f \pm S_{\alpha}f \in {\mathcal M}_{\alpha}^{\pm}$. Therefore, the decomposition
\[
f = \frac{1}{2} (f + S_{\alpha}f) + \frac{1}{2} (f - S_{\alpha}f)
\]
corresponds to (\ref{decomp_domA+}). This implies the representation (\ref{P_alpha_pm}) of $P_{\alpha}^{\pm}f$ and hence,
\[
J_{\alpha}f = \frac{1}{2} (f + S_{\alpha}f) - \frac{1}{2} (f - S_{\alpha}f) = S_{\alpha}f.
\]
Finally, by Lemma \ref{lem_int_exist}(v) and Lemma \ref{lem_Mpm}(ii) we obtain
\[
\st_{\alpha}[f,g] 
= \st_{\alpha}(S_{\alpha}f,g) 
= (Q_{\alpha}S_{\alpha}f,g)_r =  (JS_{\alpha}S_{\alpha}f,g)_r = [f,g]_r
\]
using the fundamental symmetry $J$ in $(L^2_r({\mathbb R}), \, [\cdot,\cdot]_r)$ 
from (\ref{fundSymm}).
\end{proof}
The indefinite inner product $\st_{\alpha} [\cdot,\cdot]$ can be characterized more explicitly:
\begin{lem}\label{lem_improper_integral}
For 
$f, g \in \dom \st_{\alpha}$ the following limit exists and we have
\begin{equation}\label{improper_integral}
\lim_{k \rightarrow \infty} \int_{-k}^{k} f \overline{g} \, r  \, dx
= \st_{\alpha}[f,g].
\end{equation}
\end{lem}
\begin{proof}
By Lemma \ref{lem_Mpm}(vii) we know that
$f_k := \chi_{[-k,k]} f, \, g_k := \chi_{[-k,k]} g \in L_{r, 0}^2$ satisfy 
$f_k \longrightarrow f$ and
$g_k \longrightarrow g \; (k \longrightarrow \infty)$ with respect to $\st_{\alpha} (\cdot,\cdot)$.
Since $\st_{\alpha} [\cdot,\cdot]$ is continuous  with respect to $\st_{\alpha} (\cdot,\cdot)$ this implies
\[
\st_{\alpha}[f,g] = \lim_{k \rightarrow \infty} \st_{\alpha} [f_k,g_k]
= \lim_{k \rightarrow \infty} [f_k,g_k]_r
= \lim_{k \rightarrow \infty} \int_{-k}^{k} f \overline{g} \, r  \, dx
\]
by (\ref{t_alpha_on_Lr0}).
\end{proof}
Finalle, we collect some of the above results:
\begin{prop}\label{prop_t_alpha_KreinSpace}
For $\alpha \in [0,2]$ the space
$\dom \st_{\alpha}$ from (\ref{dom_t_alpha}) is a Kre\u{\i}n space with the indefinite inner product
$\st_{\alpha}[\cdot,\cdot]$ satisfying (\ref{improper_integral}).
A fundamental decomposition is given by (\ref{decomp_dom_t_alpha})
and $\st_{\alpha} (\cdot,\cdot)$ from (\ref{t_alpha_pos}) is the corresponding Hilbert space inner product.
\end{prop}

\subsection{Conclusions for forms associated with the multiplication operator}\label{subsec_conclusions_mult_op}

Now, we bring together the results from the Sections \ref{subsec_example_regular} and \ref{subsec_example_nonregular}.
We start with the ``Kre\u{\i}n space point of view''.
\begin{thm}\label{thm_conclusions_KreinSpaceCompletion}
Let $T$ be the self-adjoint multipication operator in the Hilbert space $(L^2_{r_-}({\mathbb R}), \, (\cdot,\cdot)_{r_-})$ according to (\ref{op_T})
and let ${\mathcal N}_T$ be given according to Theorem \ref{thm_1-1_only_T}.
Furthermore, for $\alpha \in [0,2]$ consider the space
$\dom \st_{\alpha}$ from (\ref{dom_t_alpha}) and the inner products
$\st_{\alpha}[\cdot,\cdot]$, $\st_{\alpha}(\cdot,\cdot)$ and $\st [\cdot,\cdot]$ satisfying (\ref{improper_integral}), (\ref{t_alpha_pos}) and (\ref{form_t}), respectively.
Then the following statements hold true:
\begin{enumerate}
\item[(i)] The form $\st [\cdot,\cdot]$ 
coincides with the form $\st_{0} [\cdot,\cdot]$, i.e $\st_{\alpha} [\cdot,\cdot]$ 
with $\alpha = 0$.
\item[(ii)] All spaces $(\dom \st_{\alpha}, \, \st_{\alpha} [\cdot,\cdot])$ with $\alpha \in [0,2]$ are Kre\u{\i}n space completions of $(\dom T, \, (T\cdot,\cdot)_{r_-})$ which are continuously embedded in $(L^2_{r_-}({\mathbb R}), \, (\cdot,\cdot)_{r_-})$.
An associated Hilbert space inner product is given by $\st_{\alpha} (\cdot,\cdot)$.
\item[(iii)] For all $\alpha \in (0,2]$ the J-self-adjoint, J-non-negative and boundedly invertible range restriction 
$A_{\st_{\alpha}}$ of $T$ to $\dom \st_{\alpha}$ is given by
\[
\dom A_{\st_{\alpha}} = \{f \in L^2_{r_+}({\mathbb R}) \, | \, Tf \in \dom \st_{\alpha} \}, \quad
(A_{\st_{\alpha}}f)(x) = x f(x)
\]
and infinity is a singular critical point of
$A_{\st_{\alpha}}$ in $(\dom \st_{\alpha}, \, \st_{\alpha} [\cdot,\cdot])$.
\item[(iv)]  For $\alpha = 0$ infinity is not a singular critical point of $A_{\st_0} = A$ in $(L^2_{r}({\mathbb R}), \, [\cdot,\cdot]_{r})$
from Examle \ref{ex_model_spaces}.
\item[(v)] $(L^2_{r}({\mathbb R}), \, [\cdot,\cdot]_{r}, \, A) \in {\mathcal N}_T$ from Examle \ref{ex_model_spaces}
is the regularization of all elements $(\dom \st_{\alpha}, \, \st_{\alpha} [\cdot,\cdot], \, A_{\st_{\alpha}}) \in {\mathcal N}_T$ with $\alpha \in (0,2]$.
\end{enumerate}
\end{thm}
\begin{proof}
(i) follows from Lemma \ref{lem_extremal}(iv) and Lemma \ref{lem_improper_integral}.

(ii) By Proposition \ref{prop_Hilbert_space} the space $L^2_{r_+}({\mathbb R}) \, (= \dom T)$
is dense in the Kre\u{\i}n space $(\dom \st_{\alpha}, \, \st_{\alpha} [\cdot,\cdot])$
and by Lemma \ref{lem_improper_integral} on $\dom T$
the forms $\st_{\alpha} [\cdot,\cdot]$ and $[\cdot,\cdot]_r$ and hence, also $(T\cdot,\cdot)_{r_-}$ coincide.
The continuity of the embedding was also already shown in Proposition \ref{prop_Hilbert_space}.

(iv) was already observed in Section \ref{subsec_example_regular}.

(v) and (iii) are immediate consequences of (iv) by Theorem \ref{thm_exceptual}(ii).
\end{proof}
In terms of form closures the above results have the following form:
\begin{cor}\label{cor_ex_non-neg_closure}
Let $T$ be the self-adjoint multipication operator in the Hilbert space $(L^2_{r_-}({\mathbb R}), \, (\cdot,\cdot)_{r_-})$ according to (\ref{op_T}).
Furthermore, 
let the functions $f_{\tau}, \, g_{\tau}$ be given by (\ref{f_g_tau}) for $\tau \in [0,2]$.
Then the following statements hold true:
\begin{enumerate}
\item[(i)] All forms $\st_{\alpha} [\cdot,\cdot]$ 
with $\alpha \in [0,2]$
are closures of $(T\cdot,\cdot)_{r_-}$ in the Hilbert space $(L^2_{r_-}({\mathbb R}), \, (\cdot,\cdot)_{r_-})$ with gap point $0$.
\item[(ii)] The forms $\st_{\alpha} [\cdot,\cdot]$ with $\alpha \in (0,2]$ are not regular.
\item[(iii)] The form $\st_0 [\cdot,\cdot] \, (= \st [\cdot,\cdot])$ is regular 
and the regularization of all forms $\st_{\alpha} [\cdot,\cdot]$ with $\alpha \in (0,2]$.
\item[(iv)] For $\alpha \in (0,2]$ we have $f_{\alpha} \in \dom \st_0 \setminus \dom \st_{\alpha}$ and 
$g_{0} \in \dom \st_{\alpha} \setminus \dom \st_0$
(where $g_0 = g_{\tau}$ with $\tau = 0$).
\end{enumerate}
\end{cor}
\begin{proof}
(iv) follows from Lemma \ref{lem_set_diff}.
The other statements follow immediately from Theorem \ref{thm_conclusions_KreinSpaceCompletion} by the definitions.
\end{proof}
In the present example we have identified infinitely many elements in the sets
${\mathcal M}_T$ and ${\mathcal N}_T$ from Theorem \ref{thm_1-1_only_T}:
\begin{cor}\label{cor_ex_MT_NT}
In the setting of Theorem \ref{thm_conclusions_KreinSpaceCompletion} we have
\begin{eqnarray}
\{ \st_{\alpha} [\cdot,\cdot] \; | \; \alpha \in [0,2] \} 
& \subset &  {\mathcal M}_T,  \nonumber \\
\{ (\dom \st_{\alpha}, \, \st_{\alpha} [\cdot,\cdot], \, A_{\st_{\alpha}} ) \; | \; \alpha \in [0,2] \}
&  \subset &  {\mathcal N}_T.      \nonumber
\end{eqnarray}
For $0 \leq \alpha < \beta \leq 2$ the set differences 
$\dom \st_{\alpha} \setminus \dom \st_{\beta}$ and $\dom \st_{\beta} \setminus \dom \st_{\alpha}$
are not empty (and Lemma \ref{lem_set_diff} presents elements in these set differences).
\end{cor}

Finally, we have a look at the eigenspectral function $E_{\alpha} :=  E_{\dom \st_{\alpha}}$ of $A_{\st_{\alpha}}$
in the Kre\u{\i}n space $(\dom \st_{\alpha}, \, \st_{\alpha} [\cdot,\cdot])$.
Using again Theorem \ref{thm_spectralFunction} as in Section \ref{subsec_example_regular} 
this function is given by the restriction of the spectral measure of $T$ in (\ref{spectralMeasure})
to $\dom \st_{\alpha}$, i.e.
\begin{equation}\label{spectraFunction_alpha}
E_{\alpha}(\Delta) f = \chi_{\Delta} f \qquad (f \in \dom \st_{\alpha})
\end{equation}
for each set $\Delta \in \Sigma$.
Since for $\alpha \in (0,2]$ infinity is a singular critical point
we know that $E_{\alpha}([\varepsilon,\lambda]) f$ or $E_{\alpha}([-\lambda,-\varepsilon]) f$ 
does not converge in $(\dom \st_{\alpha}, \, \st_{\alpha} [\cdot,\cdot])$ 
with $\lambda \longrightarrow \infty$ for some $f \in \dom \st_{\alpha}$.
Now, using the function $g_{\tau}$ for $\tau = 0$
\[
g_{0}(x) = \frac{1}{\sqrt{r(x) x}}
\quad (x \in {\mathbb R} \setminus [-\varepsilon, \varepsilon])
\]
from (\ref{f_g_tau}) (with $\varepsilon > 0$ from (\ref{function_r_sign})) we can make this result more explicit.
\begin{prop}\label{prop_ex_no-converge}
For $\alpha \in (0,2]$ denote the 
norm for 
operators in
the Hilbert space $(\dom \st_{\alpha}, \, \st_{\alpha} (\cdot,\cdot))$ by $|| \cdot ||_{\alpha}$.
Then, the following statements hold true:
\begin{enumerate}
\item[(i)] For each $f \in \dom \st_{\alpha}$ we have
\[
E_{\alpha}([-k,k]) f \longrightarrow f  \quad (k \longrightarrow \infty)
\]
with convergence in $(\dom \st_{\alpha}, \, \st_{\alpha} (\cdot,\cdot))$.
\item[(ii)] $E_{\alpha}((\varepsilon,k]) g_{0}$ does not converge in this space
for $k \longrightarrow \infty$.
More precisely, we have
\begin{equation}\label{g_0_unbounded}
g_{0} \in \dom \st_{\alpha}, \quad
\st_{\alpha} (E_{\alpha}((\varepsilon,k]) g_{0}, E_{\alpha}((\varepsilon,k]) g_{0})
\longrightarrow \infty \quad (k \longrightarrow \infty).
\end{equation}
\item[(iii)] For all $k \in {\mathbb N}, \, k > \varepsilon$ we have
\[
|| \, E_{\alpha}([-k,k]) \, ||_{\alpha} = 1, \qquad
|| \, E_{\alpha}((\varepsilon,k]) \, ||_{\alpha} \geq \frac{2 (\sqrt{k^{\alpha}} - \sqrt{\varepsilon^{\alpha}})}{\alpha (g_0,g_0)_{\eta_{\alpha}}}.
\]
\end{enumerate}
\end{prop}
\begin{proof}
(i) was already shown in Lemma \ref{lem_Mpm}(vii).

(ii) In Corollary \ref{cor_ex_non-neg_closure}(iv) (or Lemma \ref{lem_set_diff})
it was already mentioned that
$g_{0} \in \dom \st_{\alpha}$.
For the even function $g_0$
we introduce the notation $g_k := \chi_{(\varepsilon,k]} g_{0}$ 
and
\begin{eqnarray}
g_{k,o}(x) & := &  \frac{1}{2} (\chi_{(\varepsilon,k]}(x) g_{0}(x) - \chi_{(\varepsilon,k]}(-x) g_{0}(-x))     
\nonumber \\
& = & 
\frac{1}{2} (\chi_{(\varepsilon,k]}(x) g_{0}(x) - \chi_{[-k,-\varepsilon)}(x) g_{0}(x))  \qquad (x \in {\mathbb R})     \nonumber
\end{eqnarray}
for the odd part of $g_k$.
Then, 
we can estimate
\begin{eqnarray}
\st_{\alpha} (g_{k}, g_{k}) & = & 2 (g_{k,o},g_{k,o})_{\omega_{\alpha}} + (g_k,g_k)_{\eta_{\alpha}} \geq 2 (g_{k,o},g_{k,o})_{\omega_{\alpha}}     \nonumber \\
& = & \frac{1}{2} \int_{-k}^{-\varepsilon} |g_0|^2  \, \omega_{\alpha} \, dx
+ \frac{1}{2} \int_{\varepsilon}^k |g_0|^2  \, \omega_{\alpha} \, dx                  \nonumber \\
& = & \frac{1}{2} \int_{-k}^{-\varepsilon} |x|^{\frac{\alpha -2}{2}}  \, dx
+ \frac{1}{2} \int_{\varepsilon}^k |x|^{\frac{\alpha -2}{2}}  \, dx             
\nonumber    \\
& = &
\frac{2}{\alpha} (\sqrt{k^{\alpha}} - \sqrt{\varepsilon^{\alpha}})   
\quad \longrightarrow \infty   \qquad (k \longrightarrow \infty) .           \label{sqrt_k}
\end{eqnarray}
This implies (\ref{g_0_unbounded}).

(iii) For $f \in \dom \st_{\alpha}$ we can estimate
\[
\st_{\alpha} (E_{\alpha}([-k,k]) f, E_{\alpha}([-k,k]) f)
= 2 \int_{-k}^{k} |f_o|^2  \, \omega_{\alpha} \, dx + \int_{-k}^{k} |f|^2  \, \eta_{\alpha} \, dx
\leq \st_{\alpha} (f,f)
\]
and here, we have ``$=$'' if $f = \chi_{[-k,k]} f$. This implies $|| E_{\alpha}([-k,k]) ||_{\alpha} = 1$.
With $g_k = E_{\alpha}((\varepsilon,k]) g_0$ from (ii) we can further estimate
\[
\st_{\alpha} (g_{k}, g_{k}) \leq || \, E_{\alpha}((\varepsilon,k]) \, ||_{\alpha} \st_{\alpha} (g_{0}, g_{0}) .
\]
Therefore, using again the calculation in (\ref{sqrt_k}) we have
\[
|| \, E_{\alpha}((\varepsilon,k]) \, ||_{\alpha} 
\geq \frac{\st_{\alpha} (g_{k}, g_{k})}{\st_{\alpha} (g_{0}, g_{0}) }
\geq \frac{2 (\sqrt{k^{\alpha}} - \sqrt{\varepsilon^{\alpha}})}{\alpha (g_0,g_0)_{\eta_{\alpha}}}.
\]
since $g_0$ is even and hence $\st_{\alpha} (g_{0}, g_{0}) = (g_0,g_0)_{\eta_{\alpha}}$.
\end{proof}
Proposition \ref{prop_ex_no-converge}(iii) can be regarded as an addition (and partly sharpening)
for the norm estimates of the eigenspectral function from \cite{CGL}.
Note that Proposition \ref{prop_ex_no-converge} studies a similar question as at the end of Example \ref{ex_SturmLiouville}
and (in contrast to Example \ref{ex_SturmLiouville}) indeed, gives an answer in the present setting.

Furthermore, it should be mentioned that in the definition of the functions $\eta_{\alpha}$ and $\omega_{\alpha}$
the term $|x|^{\alpha}$ can also be replaced by a more general even function.
This observation shows that in Corollary \ref{cor_ex_MT_NT} we have ``$\subseteq$'' but not ``$=$''.
In other words, there are many other non-regular closures of
$(T\cdot,\cdot)_{r_-}$ in the Hilbert space $(L^2_{r_-}({\mathbb R}), \, (\cdot,\cdot)_{r_-})$.
Therefore, the value $\alpha \in [0,2]$ is certainly not the only information
which is stored in the closures of $(T\cdot,\cdot)_{r_-}$ in addition to the information of $T$ itself.


\subsection*{Acknowledgment}
The author wants to thank Henk de Snoo and Branko \'Curgus for some helpful hints.

\subsection*{Funding}
No funding was received to assist with the preparation of this manuscript.

\subsection*{Competing Interests}
The author has no relevant financial or non-financial interests to disclose.

\subsection*{Data Availability}
The author declares that the data supporting the findings of this study are available within the paper.

\end{document}